\documentclass[10pt,a4paper,reqno]{amsart} 
\usepackage{latexsym}
\usepackage{verbatim}
\usepackage{epsfig}
\usepackage{rotating}
\usepackage{amssymb}
\usepackage[T1]{fontenc}
\usepackage{afterpage}
\usepackage{color}
\usepackage{dsfont}
\usepackage{url}
\usepackage[utf8]{inputenc}
\usepackage{amssymb,amsfonts,amsmath,stmaryrd,bbm}
\usepackage[plainpages=false,pdfpagelabels,colorlinks=true,citecolor=blue,hypertexnames=false]{hyperref}
\catcode`\@=11
\def\section{\@startsection{section}{1}%
 \z@{.7\linespacing\@plus\linespacing}{.5\linespacing}%
 {\normalfont\bfseries\scshape\centering}}

\def\subsection{\@startsection{subsection}{2}%
  \z@{.5\linespacing\@plus\linespacing}{.5\linespacing}%
  {\normalfont\bfseries\scshape}}

\def\subsubsection{\@startsection{subsubsection}{3}%
 \z@{.5\linespacing\@plus\linespacing}{-.5em}%{.5\linespacing}%
 {\normalfont\bfseries}}
\catcode`\@=12

\newtheorem{Theorem}{Theorem}[section]
\newtheorem{Lemma}[Theorem]{Lemma}
\newtheorem{Proposition}[Theorem]{Proposition}
\newtheorem{Corollary}[Theorem]{Corollary}

\theoremstyle{definition}
\newtheorem{Definition}[Theorem]{Definition}
\newtheorem{Remark}[Theorem]{Remark}

\newtheorem{Example}[Theorem]{Example}

\def\qee{$\hfill{\Box}$}

\newcommand\bi[2]{{{#1}\atopwithdelims(){#2}}}

\newfont{\bbold}{msbm10 scaled \magstep1}
\newfont{\bbolds}{msbm7 scaled \magstep1}

\newcommand{\ns}{\mathbb{N}}%{\mbox{\bbold N}}

\newcommand{\zs}{\mathbb{Z}}%{\mbox{\bbold Z}}

\newcommand{\rs}{\mathbb{R}}%{\mbox{\bbold R}}
\newcommand{\JJ}{J}

\newcommand{\JI}{\mathcal J}

\newcommand{\UU}{\mathcal J}

\newcommand{\jj}{j}

\newcommand{\hh}{h}

\newcommand{\cj}{\mathfrak{j}}
\newcommand{\ch}{\mathfrak{h}}

\newcommand{\ctS}{\widetilde{\mathcal S}}

\newcommand{\cS}{\mathcal S}
\newcommand{\cC}{\mathcal C}

\newcommand{\cM}{\mathcal M}
\newcommand{\cP}{\mathcal P}
\newcommand{\cT}{\mathcal T}

\newcommand{\cH}{\mathcal H}

\newcommand{\beq}{\begin{equation}}
\newcommand{\eeq}{\end{equation}}

\newcommand{\gf}{generating function}

\def\emm#1,{{\em #1}}

\newcommand{\la}{\lambda}

\newcommand{\tA}{\widetilde A}
\newcommand{\tS}{\widetilde S}

\renewcommand{\H}{\mathcal{H}}
\newcommand{\Ht}{\widetilde{\H}}
\newcommand{\SP}{\H_1}
\newcommand{\seq}{\mathbf{s}}
\newcommand{\Max}{\operatorname{Max}}
\newcommand{\Min}{\operatorname{Min}}
\newcommand{\inv}{\operatorname{inv}}
\newcommand{\area}{\operatorname{area}}
\newcommand{\bijheapscycle}{\operatorname{\psi}}
\newcommand{\Col}{\operatorname{Col}}
\newcommand{\BijDiagramPyramid}{\operatorname{\Upsilon}}

\newcommand{\Cylset}{\mathcal{O}}

\begin{document}
\title{321-avoiding affine permutations and their many heaps}

\author[R. Biagioli]{Riccardo Biagioli}
\address{R. Biagioli, F, Jouhet, P. Nadeau: Univ Lyon, Universit\'e Claude Bernard Lyon 1, CNRS UMR 5208, Institut Camille Jordan, 43 blvd. du 11 novembre 1918, F-69622 Villeurbanne cedex, France}
\email{biagioli@math.univ-lyon1.fr, jouhet@math.univ-lyon1.fr, nadeau@math.univ-lyon1.fr}

\author[F. Jouhet]{Frédéric Jouhet}
%\address{FJ: }
%\email{}
%
\author[P. Nadeau]{Philippe Nadeau}
%\address{PN: }
%\email{}

\begin{abstract}
We study $321$-avoiding affine permutations, and prove a formula for their enumeration with respect to the inversion number by using a combinatorial approach. This is done in two different ways, both related to Viennot's theory of heaps. First, we encode these permutations using certain heaps of monomers and dimers. This method specializes to the case of affine involutions. For the second proof, we introduce periodic parallelogram polyominoes, which are new combinatorial objects of independent interest. We enumerate them by extending the approach of Bousquet-M\'elou and Viennot used for classical parallelogram polyominoes. We finally establish a connection between these new objects and $321$-avoiding affine permutations.
\end{abstract}
\date{\today}

\maketitle

%\tableofcontents

%%%%%%%%%%%%%%%%%%%%%%%%%%%%%%%%%%%%%%%%%%%%%%%%%%%%%%%%%%%%%%
\section{Introduction}\label{sec:intro}
%%%%%%%%%%%%%%%%%%%%%%%%%%%%%%%%%%%%%%%%%%%%%%%%%%%%%%%%%%

The symmetric group $S_n$ can be viewed as the Coxeter group of type $A_{n-1}$. In this correspondence, the Coxeter length of the permutation is the inversion number.  Among permutations, those that avoid the pattern $321$ are of great interest in combinatorics and algebra. They are known to be counted by the $n$th Catalan number. From an algebraic point of view, Billey, Jockusch, and Stanley  showed in~\cite{BJS} that a permutation is $321$-avoiding if, and only if its corresponding element in the Coxeter group of type $A$ is fully commutative (FC), which means that any two of its reduced decompositions are related by a series of transpositions of adjacent commuting generators. These FC elements also naturally index a linear basis of the Temperley--Lieb algebra associated with the Coxeter group of type $A_{n-1}$.
 
 These considerations can be lifted to the affine case. A result of Green~\cite{Gre321} (independently rediscovered by Lam in~\cite{Lam}) shows that FC elements in the affine Coxeter group of type $\tA_{n-1}$ are also characterized to be $321$-avoiding, once interpreted as infinite (or affine) permutations (see Section~\ref{sec:321}, where precise definitions regarding these permutations are recalled). Here again, the length corresponds to the inversion number. 
Algebraically, $321$-avoiding affine permutations  are connected with the affine case of Stanley's symmetric functions defined in~\cite{Stanley}. More precisely, it is shown by Lam~\cite{Lam} that the affine  Stanley symmetric function $\tilde{F}_w$ associated with any $321$-avoiding affine permutation $w$ is equal to a cylindric skew  Schur function, which is also proved to be a skew affine Schur function. Such generalizations of the classical symmetric Schur functions were actually introduced by Postnikov in~\cite{Postnikov}, where the connection was established with the so-called affine nil Temperley--Lieb algebra. Postnikov also observed that combinatorics on cylindric (skew) Schur functions can describe a quantum cohomology of the Grassmannian. 

There is an infinite number of $321$-avoiding affine permutations of a given size, so we calculate how many of them have a fixed inversion number. More generally, for any Coxeter group $W$, it is interesting to compute the generating function $W^{FC}(q)$ for FC elements, where $q$ records the Coxeter length. Algebraically, this yields information on the growth of the associated generalized Temperley--Lieb algebra $TL(W)$ defined by Graham in~\cite{Graham}, or equivalently the Hilbert series of the associated graded nil Temperley--Lieb algebra.
By an approach involving families of lattice paths, recursive expressions for these series were given in~\cite{BJN-long}, for all finite and affine Coxeter groups. As a consequence, it was proved that the associated generalized Temperley--Lieb algebra has at most linear growth when the Coxeter group is irreducible and affine. For any classical finite or affine family of Coxeter groups $(W_n)_n$, one naturally introduces the bivariate generating function in $x$ and $q$
\begin{equation}
\label{eq:bivariate}
\sum_n W_n^{FC}(q)x^n.
\end{equation}
 In~\cite{BBJN}, explicit expressions for these series are computed by recursive methods. The counterpart for involutions was also treated there, while the lattice path point of view for them was examined in~\cite{BJN-inv}.

 By using $321$-avoiding permutations, Barcucci \emph{et al.} proved~\cite{barcucci}  in the type $A$ case an elegant explicit expression for the above bivariate generating function, as a $q$-logarithmic derivative of a $q$-Bessel type series $J(x)$: this is the first formula in Theorem~\ref{thm:A-tA} of Section~\ref{sec:321} below. In~\cite{BBJN}, the recursive methods also yielded a simple expression in type $\tA$, the striking fact being that this time the logarithmic derivative (with respect to $x$) of the same series $J(x)$ occurs: this is the second formula of Theorem~\ref{thm:A-tA}. Moreover, it was also proved that for $321$-avoiding (affine) involutions arise the same kinds of expressions as ($q$-)logarithmic derivatives of a simpler $q$-hypergeometric series $\mathcal{J}(x)$, see Theorem~\ref{thm:A-tA-inv}.\\

The main motivation of the present paper is to provide a combinatorial framework explaining bijectively Theorems~\ref{thm:A-tA} and~\ref{thm:A-tA-inv}. To this aim, we will introduce in Section~\ref{sec:321} particular posets, which we call affine alternating diagrams: our starting point will then be a bijection, expressed in Theorem~\ref{thm:bijectionSD}, between them and $321$-avoiding affine permutations. This bijection restricts nicely both to finite permutations and to involutions. These diagrams have been studied under other names: they arise as a way of encoding the whole commutation class for FC elements in affine type $\tA$, in the spirit of the initial work on FC elements by Stembridge~\cite{St3}, Green~\cite{Gre321}, Hagiwara~\cite{HagiwaraAtilde}, or the authors~\cite{BJN-long}. They also correspond essentially to the skew cylindric shapes that index cylindric skew Schur functions, though our representation is slightly different~\cite{Lam}.

 Once $321$-avoiding affine permutations are interpreted in terms of affine alternating diagrams, we will give two different combinatorial approaches towards the proofs of the generating functions mentioned above, both based on Viennot's general theory of heaps~\cite{Viennot1}. Section~\ref{sec:heaps} is therefore devoted to a collection of definitions and properties on heaps of pieces and cycles, together with a proof of the so-called Inversion Lemma (namely Lemma~\ref{lem:inversion}) and its adaptation to the enumeration of pyramids (see Corollary~\ref{Cor:enumpyramids}).

In Section~\ref{sec:heaps-monomeres-dimeres}, we will show how to use heaps of cycles and transform them in our case in terms of particular heaps of monomers and dimers: the main bijective result is given in Theorem~\ref{thm:diagramstopyramids}. Thanks to the Inversion Lemma, we will show how the enumeration of $321$-avoiding affine permutations will boil down to enumerating trivial heaps of monomers and dimers satisfying some specific conditions (see Theorem~\ref{theorem:trivialh}). By this approach, we will derive the generating functions  for finite and affine $321$-avoiding permutations of Theorem~\ref{thm:A-tA}, and their counterpart for involutions of Theorem~\ref{thm:A-tA-inv}. 

Our second bijective approach is detailed in Section~\ref{sec:parallelogram}, where we define a new family of combinatorial objects that we call periodic parallelogram polyominoes (PPPs) (see also~\cite{BBJN-bij}). Inspired by the seminal work of Bousquet-M\'elou and Viennot~\cite{mbm-viennot} on parallelogram polyominoes, we will prove in Proposition~\ref{prop:bijcyclicheap} that PPPs are in bijection with a set of heaps of segments satisfying some specific conditions. This will enable us to derive in Theorem~\ref{thm:enumPPP} the generating function of PPPs, with respect to a trivariate weight, as a logarithmic derivative in the variable $y$, of a $q$-series $N(x,y,q)$ which was introduced in~\cite{mbm-viennot}.  We note that PPPs were defined independently and studied in~\cite{boussicault-laborde} (see also~\cite{aval-ppp}), where they are interpreted in terms of binary trees and counted according to different parameters.

 Finally, we exhibit a combinatorial interpretation of $321$-avoiding affine permutations in terms of PPPs. This was our initial motivation for introducing these objects; as a consequence, we obtain the second bijective proof of Theorem~\ref{thm:A-tA}. 
 
 In a last and short section, we will propose some combinatorial problems raised by our approach.

%%%%%%%%%%%%%%%%%%%%%%%%%%%%%%%%%%%%%%%%%%%%%%%%%%%%%%%%%%%%%%
\section{$321$-avoiding affine permutations and affine alternating diagrams}\label{sec:321}
%%%%%%%%%%%%%%%%%%%%%%%%%%%%%%%%%%%%%%%%%%%%%%%%%%%%%%%%%%

In the whole section $n$ is an integer greater than $1$. 

%%%%%%%%%%%%
\subsection[Enumeration of 321-avoiding affine permutations]{Enumeration of $321$-avoiding affine permutations}
%%%%%%%%%%%

\begin{Definition}
An \emph{affine permutation} of size $n$  is a bijective function $\sigma:\mathbb{Z}\rightarrow\mathbb{Z}$ such that $\sigma(i+n)=\sigma(i)+n$ for all $i\in\mathbb{Z}$, and  $\sum_{i=1}^n\sigma(i)=\sum_{i=1}^ni$.
\end{Definition}

 Affine permutations of size $n$  form  a group $\widetilde{S}_n$ under composition. One can write down an affine permutation through its biinfinite sequence of values $(\sigma(i))_{i\in\mathbb{Z}}$. This is the {\em complete notation} of $\sigma$. For example,
\[
\ldots \mid -14,5,-12,-9,-8,6,11,-7 \mid {\bf -6,13,-4,-1,0,14,19,1} \mid 2,21,4,9,8,22,27,9 \mid \ldots
\]
is an element of $\widetilde{S}_8$, where we highlighted the values $\sigma(1),\ldots,\sigma(8)$. Clearly any $\sigma \in \widetilde{S}_n$ is uniquely determined by its values on $\{1,\ldots,n\}$, and the expression $\sigma=[\sigma(1),\ldots,\sigma(n)]$ is usually called the {\em window notation} of $\sigma$. In our previous example, $\sigma=[-6,13,-4,-1,0,14,19,1]\in \widetilde{S}_8$. Note that if $\sigma \in S_n$, then there exists a unique $\tilde{\sigma} \in \widetilde{S}_n$ such that $\tilde{\sigma}(i)=\sigma(i)$ for $i \in \{1,\ldots,n\}$, and this allows one to consider $S_n$ as a subgroup of $\widetilde{S}_n$.

It is well-known (see \cite{LusztigTransactions, bjorner-brenti-book}) that $\widetilde{S}_n$ is a realization of the Coxeter system of type $\tA_{n-1}$ when we consider as generating set $\{s_0,s_1,\ldots,s_{n-1}\}$, where 
$$s_0=[0,2,3,\ldots,n-1,n+1] \ \mbox{and} \ s_i=[1,2\ldots,i-1,i+1,i,i+2,\ldots,n], \ i=1,\ldots,n-1.$$
We also define $s_n=s_0$ and $s_{-1}=s_{n-1}$. More generally, the indices of the generators are taken modulo $n$, which reflects the cyclic structure of the Dynkin diagram of type $\tA_{n-1}$ depicted below.

\begin{figure}[!ht]
\begin{center}
 \includegraphics[width=0.20\textwidth]{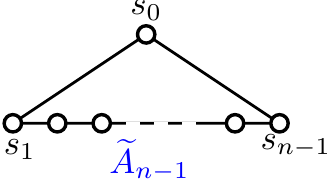}
% \caption{\label{fig:Dynkin_Atilde} The Dynkin diagram of type $\widetilde{A}_{n-1}$.}
 \end{center} 
\end{figure}

With respect to the generating set $\{s_0,s_1,\ldots,s_{n-1}\}$, there is a natural \emph{length function} $\sigma \mapsto \ell(\sigma)$ which counts the minimal number of generators needed to write $\sigma$ as a product of them. For $\sigma \in \widetilde{S}_n$, we define
\[
\inv(\sigma)=\left|\{(i,j) \in \{1,\ldots,n\} \times \mathbb{N} \mid  i<j\text{ and } \sigma(i)>\sigma(j)\}\right|.
\]
   This counts the number of {\em affine inversions} of $\sigma$. It coincides with the usual inversion number for finite permutations. Moreover, Shi \cite{Shi} showed that $\inv(\sigma)=\ell(\sigma)$ for any $\sigma \in \widetilde{S}_n$ where $\ell$ is the Coxeter length (see also \cite[\S8]{bjorner-brenti-book}). 
\smallskip

An affine permutation $\sigma$ is {\em $321$-avoiding} if there are no $i<j<k$ in $\mathbb{Z}$ satisfying $\sigma(i)>\sigma(j)>\sigma(k)$. Green showed in~\cite{Gre321} that an affine permutation is $321$-avoiding if, and only if it is \emph{fully commutative} as an element of the Coxeter system of type $\tA_{n-1}$. This generalizes the well-known result of Billey, Jockush and Stanley from~\cite{BJS} for the case of the symmetric group.  We have the following characterization~\cite{Gre321}, \cite[Prop. 2.1]{BJN-long}.

\begin{Proposition}
\label{prop:caracterisation_Atilde} An affine permutation $\sigma \in \widetilde{S}_n$ is 321-avoiding if, and only if, in  any reduced decomposition of $\sigma$, the occurrences of the generators $s_i$ and $s_{i+1}$ alternate for $i=0,\ldots,n-1$.
\end{Proposition}

We denote by $\tS_{n}^{(321)}$ the set of $321$-avoiding affine permutations in $\widetilde{S}_n$, and $S_{n}^{(321)}$ its subset of finite permutations. For any subset $\mathcal{E}_n$ of $\tS_{n}^{(321)}$, we define its \emph{length generating series} by:
 $$\mathcal{E}_n(q):=\sum_{\sigma\in\mathcal{E}_n}q^{\ell(\sigma)}.$$
We also recall for $n\geq0$ the {\em $q$-Pochhammer symbol} $(x;q)_n:= (1-x)(1-xq)\cdots(1-xq^{n-1})$, and we define the two series:
\beq\label{J-def}
J(x):= \sum_{n\ge 0} \frac{(-x)^n q^{\binom{n}{2}}}{(q;q)_n(xq;q)_n},
\eeq
and 
\beq\label{DD-def}
\JI(x):=\sum_{n\geq0}\frac{(-1)^{\lceil n/2\rceil} x^nq^{\binom{n}{2}}}{(q^2;q^2)_{\lfloor n/2\rfloor}}.
\eeq
 
The two following enumerative theorems were proved in~\cite{BBJN} by using recursive decompositions (actually the first formula in Theorem~\ref{thm:A-tA} was first proved in~\cite{barcucci}).  
\begin{Theorem}\label{thm:A-tA}
Let  $S(x,q)$ and $\tS(x,q)$ be the generating functions defined by
\[
S(x,q):=\sum_{n\ge 0} S_{n+1}^{(321)}(q) x^n \qquad\hbox{and} \qquad \tS(x,q):=
\sum_{n\ge 1} \tS_{n}^{(321)}(q) x^n.
\]
 Then 
\[
S(x,q)= \frac 1 {1-xq}\frac{\JJ(xq)}{\JJ(x)}
%= \frac{\jj(xq)}{\jj(x)}
\qquad\hbox{and} \qquad \tS(x,q)=- x \frac{\JJ'(x)}{\JJ(x)}- \sum_{n\ge 1}
 \frac{x^n  q^n}{1-q^n},
\]
where the derivative is taken with respect to $x$.
\end{Theorem}

The counterpart for the set $\cS_{n+1}^{(321)}$ ({\em resp.}  $\ctS_{n}^{(321)}$) of 321-avoiding ({\em resp.} affine) involutions reads as follows.

\begin{Theorem}\label{thm:A-tA-inv}
Let  $\cS(x,q)$ and $\ctS(x,q)$ be the generating functions defined by
$$
\cS(x,q):=\sum_{n\ge 0} \cS_{n+1}^{(321)}(q) x^n \qquad\hbox{and} \qquad \ctS(x,q):=
\sum_{n\ge 1} \ctS_{n}^{(321)}(q) x^n.
$$
Then
$$
\cS(x,q)= \frac{\JI(-xq)}{\JI(x)}
\qquad\hbox{and} \qquad 
\ctS(x,q)= -x\,
\frac{ \JI'(x)}{\JI(x)},
$$
where the derivative is taken with respect to $x$.
\end{Theorem}

The rest of the paper is devoted to the construction of combinatorial objects which will be used to prove bijectively the two above results.

%%%%%%%%%%%%
\subsection{Affine alternating diagrams}
%%%%%%%%%%%

Unlike~\cite{BBJN, BJN-long, BJN-inv}, to give bijective proofs of the two previous results, we will not use the framework of Coxeter groups but instead translate the correspondence of Proposition~\ref{prop:caracterisation_Atilde} in terms of simple combinatorial objects defined as follows. 

\begin{Definition}
\label{defi:alternating_diagram}
 An {\em affine alternating diagram} of rank $n$ is a poset $D$ with elements labeled by  $\{s_0,s_1,\ldots,s_{n-1}\}$, such that for $i\in \{0,\ldots,n-1\}$ the elements with labels in $\{s_i,s_{i+1}\}$ form an {\em alternating chain} (i.e. with alternating labels) $D_{\{i,i+1\}}$, and the ordering of $D$ is the transitive closure of these chains.
\end{Definition}

We denote the set of affine alternating diagrams of rank $n$ by $\widetilde{\mathcal{D}}(n)$. Note that these were called with a different name in~\cite{BJN-long}, see Remark~\ref{rem:heapoupas} at the end of this section. Equivalent structures were defined by Hagiwara in~\cite{HagiwaraAtilde}.
\smallskip

We represent the Hasse diagram of $D$ by putting all elements labeled $s_i$ in one column. To draw it in a planar manner, one duplicates the set of elements labeled by $s_0$, and uses ones copy for the depiction of the chain $D_{\{0,{1}\}}$ and one copy for $D_{\{{n-1},0\}}$. Two examples are showed in Figure~\ref{fig:AffineAlternatingDiagramSlanted}: the extremities of a dashed line correspond to the same element.  In particular, these diagrams have respectively 31 and 16 elements.
%\marginpar{rajouter figure self dual}

\begin{figure}[!ht]
\begin{center}
 \includegraphics[width=0.5\textwidth]{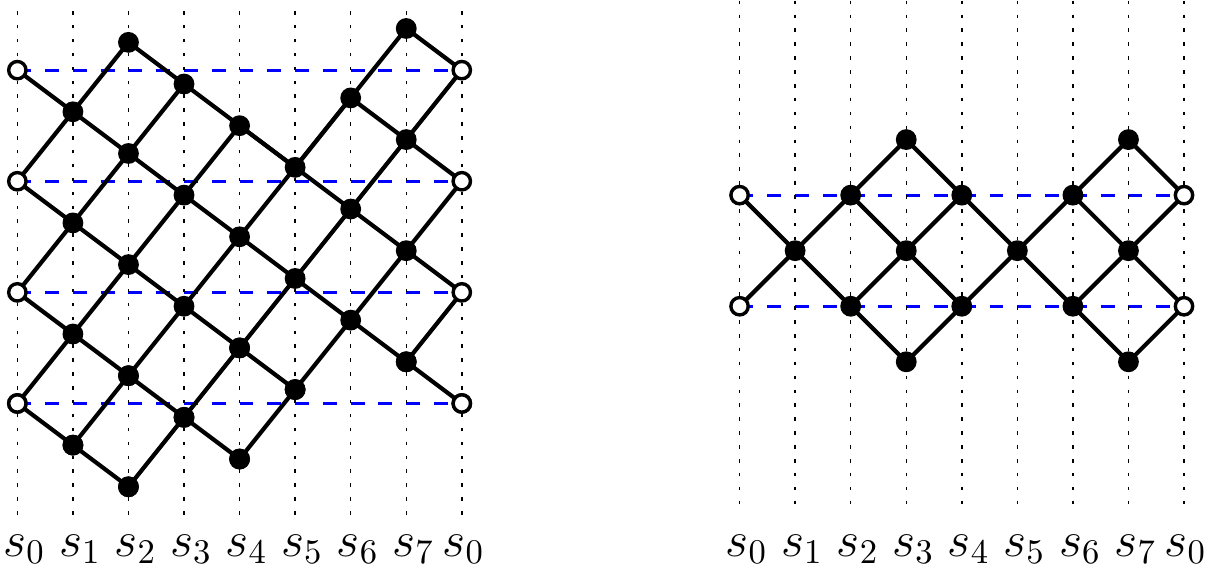}
 \caption{\label{fig:AffineAlternatingDiagramSlanted} Representations of affine  alternating diagrams of $\widetilde{S}_8$.}
 \end{center} 
\end{figure}

By definition, the alternating chains $D_{\{i,{i+1}\}}$ completely determine the poset $D$. In the following proposition we characterize which are the chains arising in an affine alternating diagram. 

\begin{Proposition}
\label{prop:caracterisation_diagrams}
For $i=0,\ldots, n-1$, let $C_{i,i+1}$ be an alternating chain labeled by $s_i$ and $s_{i+1}$. There exists an affine alternating diagrams $D$ such that $D_{i,i+1}=C_{i,i+1}$ for all $i \in \{0,\ldots,n-1\}$ if, and only if the following conditions are satisfied:
\begin{enumerate}
\item  $s_i$ appears as many times in $C_{i,i+1}$ and $C_{i-1,i}$, for all $i \in \{0,\ldots,n-1\}$; 
\item there is no $k \geq 1$ such that for all $i$, $C_{i,i+1}$ is the chain of length $2k$ with labeling $s_is_{i+1}\cdots s_is_{i+1}$ from bottom to top; \label{rectangular_up}
\item there is no $k \geq 1$ such that for all $i$, $C_{i,i+1}$ is the chain of length $2k$ with labeling $s_{i+1}s_i\cdots s_{i+1}s_i$ from bottom to top.\label{rectangular_down}
\end{enumerate}
\end{Proposition}

The set of chains described in \eqref{rectangular_up} and \eqref{rectangular_down} above clearly does not come from an affine alternating diagram. Indeed, the transitive closure of these chains violates the antisymmetry relation, as can be seen in the graphical representation in Figure~\ref{fig:rectangular_shape}: on the left ({\em resp.} right) is depicted an excluded diagram giving a set of chains of type \eqref{rectangular_up} with $k=2$ and $n=4$ ({\em resp.}  \eqref{rectangular_down} with $k=1$ and $n=4$). Such objects will play an important role in Section~\ref{subsec:BackTo321}.
\smallskip

\begin{figure}[!ht]
\begin{center}
 \includegraphics[width=0.35\textwidth]{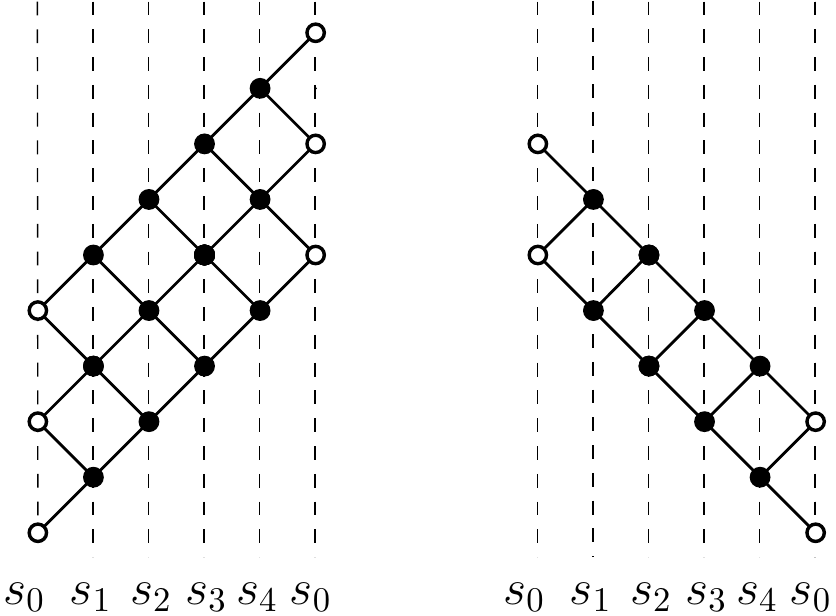}
 \caption{Two excluded diagrams: they do not represent posets as they contain cycles.~\label{fig:rectangular_shape}}
 \end{center} 
\end{figure}

The affine alternating diagrams in $\widetilde{\mathcal{D}}(n)$ containing no element labeled $s_0$ are called {\em finite alternating diagrams}. They form indeed a finite set denoted $\mathcal{D}(n)$ (see for instance~\cite{St3}).

The {\em dual} of an affine alternating diagram is the poset with the inverse order, and where the labels are kept the same. We will say that an affine alternating diagram is \emph{self-dual} if it is isomorphic to its dual.  An example is given in Figure~\ref{fig:AffineAlternatingDiagramSlanted}, right.
\smallskip

We now associate any 321-avoiding affine permutation $\sigma \in \tS^{(321)}_n$ with an affine alternating diagram via the following construction. Pick a reduced decomposition of $\sigma$. By Proposition~\ref{prop:caracterisation_Atilde}, for any $i=0,\ldots,n-1$,  the occurrences of $s_i$ and $s_{i+1}$ in such a decomposition form an alternating subword, which we identify with an alternating chain $C_{i,i+1}$.  By Proposition~\ref{prop:caracterisation_diagrams}, these chains determine a unique affine alternating diagram denoted $\Delta(\sigma)$; indeed, one easily checks that the excluded cases \eqref{rectangular_up} and \eqref{rectangular_down} can never occur. By the general theory of fully commutative elements, this construction does not depend on the choice of the reduced decomposition. 

The following theorem summarizes results of Stembridge~\cite{St3} and Green~\cite{Gre321}.

\begin{Theorem}\label{thm:bijectionSD}
The map $\Delta : \widetilde{S}_n^{(321)} \to \widetilde{\mathcal{D}}(n)$ is a bijection such that $\inv(\sigma)=|\Delta(\sigma)|$. Moreover,
\begin{itemize}
\item[$(i)$]  $\sigma \in S_n^{(321)}$  if, and only if $\Delta(\sigma) \in \mathcal{D}(n)$;
\item[$(ii)$] $\sigma$ is an involution if, and only if $\Delta(\sigma)$ is self-dual.
\end{itemize}
\end{Theorem}

There are simple graphical ways to construct the bijection $\Delta$ of the previous theorem and its inverse $\Delta^{-1}$, using the {\em line diagram} of $\sigma$, as shown in Figures~\ref{fig:alternating_diagram_000} and \ref{fig:AffineDiagramTo321}.
 
\begin{figure}[!ht]
\begin{center}
 \includegraphics[width=0.7\textwidth]{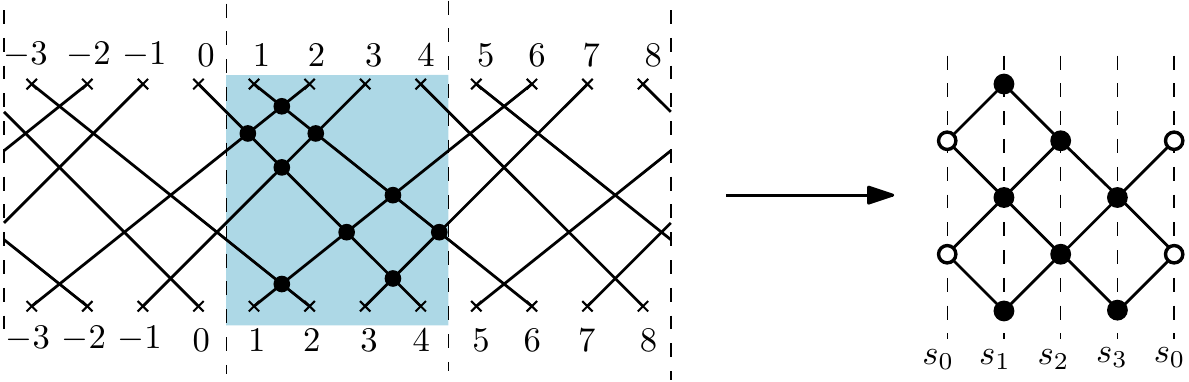}
 \caption{\label{fig:alternating_diagram_000} The image of the affine permutation  $\sigma=[6,-3,-1,8] \in \tS^{(321)}_4$ via $\Delta$.}
 \end{center} 
\end{figure}

\begin{figure}[!ht]
\begin{center}
 \includegraphics[width=0.8\textwidth]{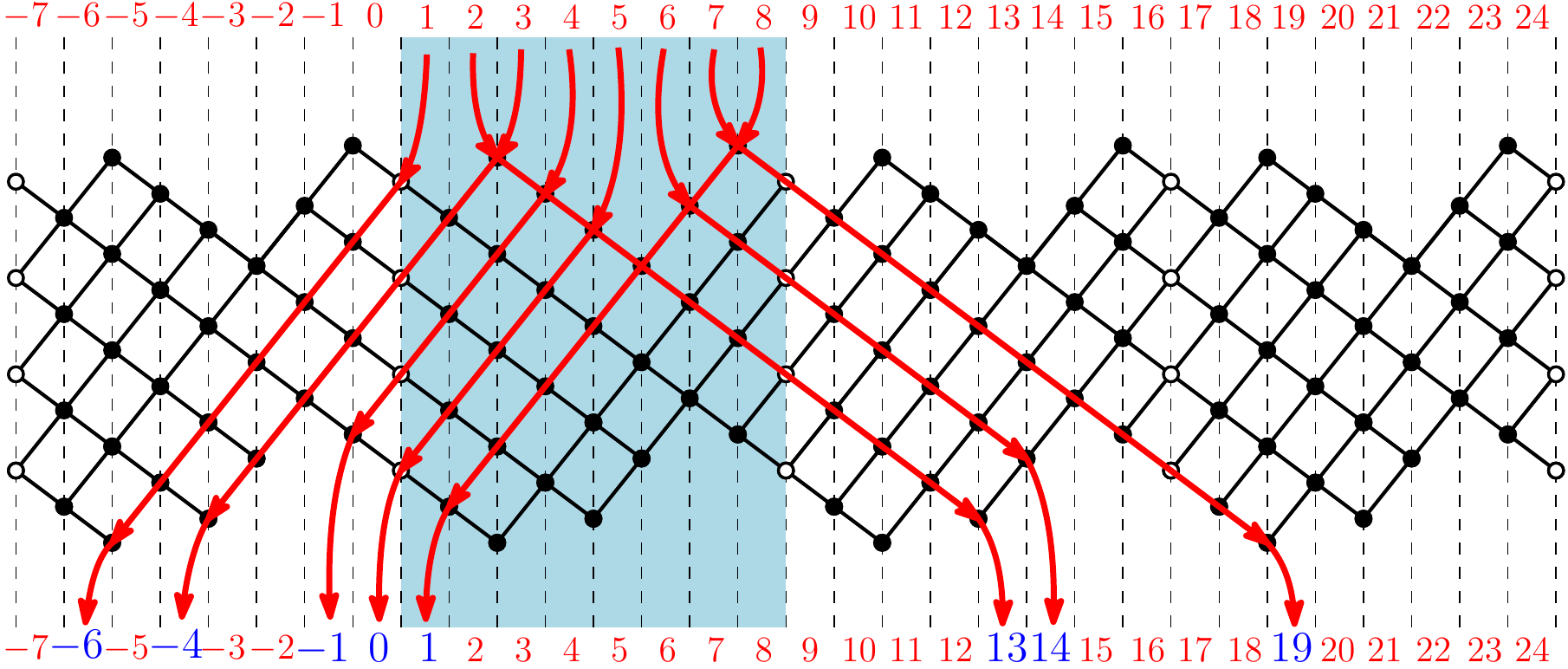}
 \caption{\label{fig:AffineDiagramTo321} The permutation $[-6, 13, -4, -1, 0, 14, 19, 1] \in \tS^{(321)}_8$ is the image via $\Delta^{-1}$ of the diagram of Figure~\ref{fig:AffineAlternatingDiagramSlanted}, left.}
 \end{center} 
\end{figure}

\begin{Remark}\label{rem:heapoupas}
In~\cite{BJN-long}, affine alternating  diagrams are called alternating heaps, in agreement with  the general definition of heaps (see next section). However the alternating condition yields many restrictions and forces these alternating heaps to be very specific. This is why the standard techniques  presented in the next section do not apply directly, and explains our decision to distinguish them by a different name.
\end{Remark}

%%%%%%%%%%%%%%%%%%%%%%%%%%%%%%%%%%%%%%%%%%%%%%%%%%%%%%%%%%%%%%
\section{Heaps of pieces}\label{sec:heaps}
%%%%%%%%%%%%%%%%%%%%%%%%%%%%%%%%%%%%%%%%%%%%%%%%%%%%%%%%%%
We recall here the theory of heaps of pieces due to Viennot~\cite{Viennot1}. We will need in particular the fundamental enumerative results from Sections~\ref{subsec:enumheaps} and~\ref{subsec:enumcycles} which will be used in Sections~\ref{sec:heaps-monomeres-dimeres} and~\ref{sec:parallelogram}.

%%%%%%%%%%%%%%%%%%%%%%%%%%%%%%%%%%%%%%%%%%%%%%%%%%%%%%%%%%
\subsection{Definitions}\label{subsec:defiheaps}
%%%%%%%%%%%%%%%%%%%%%%%%%%%%%%%%%%%%%%%%%%%%%%%%%%%%%%%%%%
Let $\cP$ be a set of \emph{basic pieces} with a symmetric and reflexive binary relation $\cC$, called the {\em concurrency relation}. The pair $(\cP,\cC)$ will be called \emph{model of heaps}.

\begin{Definition}
\label{def:labeled_heap}
 A {\em heap} is a triple $(H,\preceq, \epsilon)$, where $(H,\preceq)$ is a finite poset, and $\epsilon : H \rightarrow \cP$ is a labeling map such that:
\begin{enumerate}
\item if $x,y \in H$ and $\epsilon(x)\cC \epsilon(y)$, then either $x\preceq y$ or $y \preceq x$;
\item the relation $\preceq$ is the transitive closure of the relations from (1).
\end{enumerate}
\end{Definition}

The elements of $H$ are called {\em pieces}. We denote by $|H|$ the number of pieces in the heap $H$. When $x\preceq y$ we will say that the piece $y$ is above the piece $x$. We always consider heaps up to isomorphism, where two heaps $H_1,H_2$ are isomorphic if there is a poset isomorphism $\rho:H_1\rightarrow H_2$ that preserves the labels (i.e. such that $\epsilon_1=\epsilon_2\circ \rho$). The set of all isomorphism classes of heaps with pieces in $\cP$ and concurrency relation $\cC$ is denoted by $\cH(\cP,\cC)$. 

One can define a monoid with generators $\cP$ and relations $pp'=p'p$ whenever $p\not \!\cC \, p'$. This is called a \emph{partially commutative monoid} (or Cartier-Foata monoid, or even trace monoid). The set $\cH(\cP,\cC)$ is then in bijection with this monoid (see \cite{cartierfoata}).\smallskip

Let $H$ be a heap. A piece of $H$ is said to be {\em maximal} (\emph{resp.} {\em minimal}) if it has no piece above (\emph{resp.} below) it.  We denote by $\Max(H)$ and $\Min(H)$ the set of maximal and minimal pieces of $H$ respectively. $H$ is a {\em pyramid} if $\Max(H)$ has just one element. We denote by $\Pi(\cP,\cC)$  the set of  pyramids in $\cH(\cP,\cC)$. A {\em trivial heap}  consists of pieces which are pairwise unrelated by $\cC$. We denote by $\cT(\cP,\cC)$ the set of trivial heaps. 

In this paper we will deal with the case where $\cP$ consists of {\em segments} of the form $p=[a,b]$ with $a,b \in \ns$, $a\leq b$, and two pieces $p,p'$ of $\cP$ satisfy $[a,b]\cC [c,d]$  if  $[a,b]\cap [c,d]\neq\emptyset$. Some heaps of this type are presented in Figure~\ref{fig:example-composition}: the heap $T$ is trivial, $H_1$ is a pyramid. $H_2$ is a heap of {\em monomers} and {\em dimers}, which means that all its pieces are either points $[a]$ or segments of the form $[a,a+1]$. For any segment $p=[a,b]$, we define the {\em length} of $p$ as $\ell(p):=b-a$; more generally the length of a heap of segments is the sum of the lengths of its pieces.

\begin{figure}[!ht]
\begin{center}
 \includegraphics[width=0.8\textwidth]{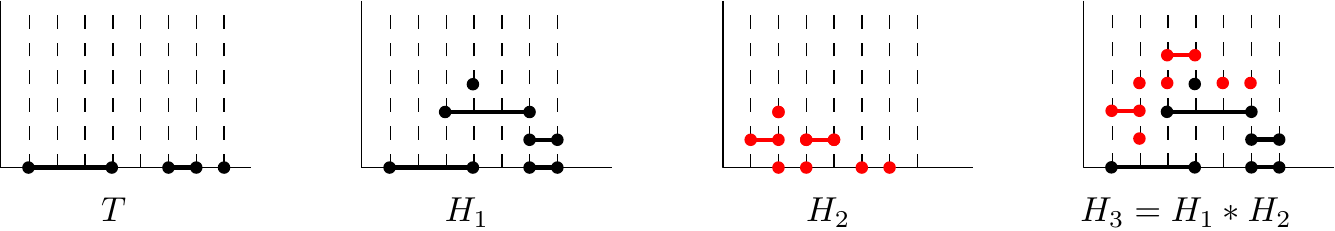}
 \caption{A trivial heap, two heaps of segments, and their composition. \label{fig:example-composition}}
 \end{center} 
\end{figure}
There is a well-known operation of {composition} of heaps, which corresponds to multiplication in the associated partially commutative monoid. Given two heaps $H_1$ and $H_2$, the \emph{composition} $H_1 * H_2$ is the heap that results by ``putting $H_2$ on top of $H_1$'', see Figure~\ref{fig:example-composition}: formally, $H_1 * H_2$ has $H_1\sqcup H_2$ as underlying set, its labeling function $\epsilon$ is defined as $\epsilon_1$ over $H_1$ and as $\epsilon_2$ over $H_2$, and its poset structure is the transitive closure of the relations of $H_1$ and $H_2$ together with $h_1\preceq h_2$ whenever $\epsilon(h_1)\cC\epsilon(h_2)$.

We finally point out that, as said in Remark~\ref{rem:heapoupas}, alternating (affine) diagrams defined in Section~\ref{sec:321} can be seen as heaps of dimers (transform points labeled $s_i$ to dimers $[i,i+1]$)  with some specific restrictions encoding the alternating condition. However the enumeration techniques that we describe in the rest of this section do not work directly on these diagrams, so we will need to transform them beforehand.

%%%%%%%%%%%%%%%%%%%%%%%%%%%%%%%%%%%%%%%%%%%%%%%%%%%%%%%%%%
\subsection{Enumeration}\label{subsec:enumheaps}
%%%%%%%%%%%%%%%%%%%%%%%%%%%%%%%%%%%%%%%%%%%%%%%%%%%%%%%%%%

The following fundamental result is due to Viennot~\cite{Viennot1} (see also~\cite[Theorem 2.1]{mbm-viennot}), and  is usually called the \emph{Inversion Lemma}. It allows to enumerate some families of heaps with respect to any weight function on their pieces. Here a \emph{weight }$v$ is a function on $\cP$ with values in a ring of formal power series.
The weight $v(H)$ of a heap $H$ is the product of the weights of all its pieces. Formally, 
\begin{equation}
\label{eq:weight}
v(H)=\prod_{x\in H} v(\epsilon(x)).
\end{equation} 
We always assume that the family of $v(H)$ for $H\in \cH(\cP,\cC)$ is summable.

\begin{Lemma}[Inversion Lemma]
\label{lem:inversion}
Let $(\cP,\cC)$ be a model of heaps and let $\cM\subseteq\cP$. Then the generating function for heaps with all maximal pieces in $\cM$ is given by
\begin{equation}
\label{eq:inversion}
\sum_{\stackrel{H \in \cH(\cP,\cC)}{\Max(H) \subseteq \cM}} v(H)=\displaystyle{\frac{\displaystyle\sum_{T \in \cT(\cP\setminus \cM, \cC)} (-1)^{|T|} v(T)}{\displaystyle\sum_{T\in \cT(\cP,\cC)} (-1)^{|T|}v(T)}}.
\end{equation} 
In particular if $\cM=\cP$, so that we enumerate all heaps, the numerator above is simply $1$.\end{Lemma}

The usefulness of this lemma is that trivial heaps form a simple family for which one can obtain formulas. Let us recall Viennot's proof, since we will use the same idea in Section~\ref{sec:parallelogram} in order to count another family of heaps. 

\begin{proof}
In this proof, we note $\cH=\cH(\cP,\cC)$ and $\cT=\cT(\cP,\cC)$. Define $\phi:\cT \times \cH \to  \cH$ by $\phi(T,H):=T*H$. In words, $\phi$ adds new minimal elements at the bottom of a heap. The idea is to use $\phi$ for a double counting of $\cT\times \cH$. Given any family of heaps $\H'\subseteq \H$, we have
\begin{equation}
\label{eq:doublecounting}
\left(\sum_{T\in \cT} (-1)^{|T|}v(T)\right)\left(\sum_{H\in \H'} v(H)\right)=
\sum_{H_0\in \phi(\cT\times \H')}v(H_0)\left(\sum_{T\in\cT, H\in\H':\,\phi(T,H)=H_0}(-1)^{|T|}\right).
\end{equation}
Consider $T,H$ as in the last inner sum. Clearly the pieces of $T$ become minima in $H_0$ via $\phi$, and  $H$ is obtained from $H_0$ by removing these minima. We write $H=H_0\setminus T$ in this case. Therefore the right-hand side of~\eqref{eq:doublecounting} can be rewritten as
\begin{equation}
\label{eq:secondterm}
\sum_{H_0\in \phi( \cT\times\H')}v(H_0)\left(\sum_{T\subseteq\Min(H_0):\, H_0\setminus T\in \H'}(-1)^{|T|}\right).
\end{equation}
%Up to this point the proof holds for any family $\cH'$ of heaps. 

We now assume $\H'$ is the set $\H_\cM$ of heaps whose maximal pieces belong to $\cM$. Comparing~\eqref{eq:inversion} and~\eqref{eq:doublecounting}, we have to prove that~\eqref{eq:secondterm} is precisely the numerator in~\eqref{eq:inversion}.  

First we remark that $H_0\in\phi(\cT \times \H_\cM)$ if, and only if $\{x\in \Max(H_0):\, \epsilon(x)\notin\cM\}\subseteq \Min(H_0)$. In this case, for a given $T\subseteq\Min(H_0)$, we have $H_0\setminus T\in \H_\cM$ if, and only if $\{x\in \Max(H_0):\, \epsilon(x)\notin\cM\}\subseteq T$. The inner sum in \eqref{eq:secondterm} vanishes by inclusion-exclusion, except when $\{x\in \Max(H_0):\, \epsilon(x)\notin\cM\}=\Min(H_0)$, under which condition the sum is equal to $(-1)^{|\Min(H_0)|}$. This case is equivalent to $H_0\in\cT(\cP\setminus \cM, \cC)$, so that $H_0=\Min(H_0)$ and \eqref{eq:secondterm} is indeed the numerator in \eqref{eq:inversion} as wanted.
\end{proof}

We can use this result to enumerate the set of pyramids.

\begin{Corollary}[Enumeration of pyramids]
\label{Cor:enumpyramids} 
Let $(\cP,\cC)$ be a model of heaps. The generating function of pyramids is given by
\begin{equation}
\label{eq:enumpyramids}
\sum_{H \in \Pi(\cP,\cC)} v(H)=\displaystyle{-\frac{\displaystyle \sum_{T \in \cT(\cP,\cC)} (-1)^{|T|} |T|v(T)}{\displaystyle\sum_{T\in \cT(\cP,\cC)} (-1)^{|T|}v(T)}}.
\end{equation}
\end{Corollary}

\begin{proof}
Let $\Pi=\Pi(\cP,\cC)$ and $\cT=\cT(\cP,\cC)$. $\Pi$ is the disjoint union  over all $M\in \cP$ of heaps such that $\Max(H)=\{M\}$, so we get 
\[
\sum_{H \in \Pi}v(H)=\sum_{M\in\cP}\left(\sum_{H:\, \Max(H)\subseteq \{M\}}v(H)-1\right),
\]
in which the term $-1$ removes the empty heap corresponding to the case $\Max(H)=\emptyset$ . By Lemma~\ref{lem:inversion}, we get 
\[
\sum_{H:\, \Max(H)\subseteq \{M\}}v(H)-1
=-\frac{\displaystyle\sum_{T\in \cT:\, M\in T} (-1)^{|T|}v(T)}{\displaystyle\sum_{T\in \cT} (-1)^{|T|}v(T)}.
\]
Summing over all $M$, and exchanging the summations gives~\eqref{eq:enumpyramids}.
\end{proof}

%This corollary can also be proved 

%The start of the proof is the same as before, and we are naturally led to show that the numerator in the last fraction is equal to \eqref{eq:secondterm} with $\cH'=\Pi$. To characterize  $\phi(\Pi\times \cT)$, we distinguish two cases.
%
% First,  $\phi(\Pi\times \cT)\setminus\cT$ consists of heaps $H_0$ such that $\Max(H_0)\setminus \Min(H_0)$ has just one element $M$, and in this case if $T\subseteq \Min(H_0)$ then $H_0\setminus T\in \Pi$ if, and only if $\Max{H_0}setminus \{M\}\subseteq T$. The inner sum in \eqref{eq:secondterm} always vanishes by inclusion in this case: indeed we cannot have  $\Max{H_0}setminus \{M\}=\Min{H_0}$ because $H_0$ is assumed to be nontrivial.
% 
% Then, all (nonempty) trivial heaps belong to  $\phi(\Pi\times \cT)$, and such heaps $H_0$  are obtained via $\phi$ from a pyramid consisting of a single piece. It follows that the inner sum in this case is equal to $|H_0|(-1)^{|H_0|-1}$, and putting things together we obtain indeed the numerator in~ \eqref{eq:enumpyramids}.

%
%$$\sum_{P \in \Pi(\cP,\cC)} \frac{1}{|P|} v(P)=-{\rm log} \left(\sum_{T \in \cT)} (-1)^{|T|} v(T) \right),
%$$

%%%%%%%%%%%%%%%%%%%%%%%%%%%%%%%%%%%%%%%%%%%%%%%%%%%%%%%%%%%%%%
\subsection{Heaps of cycles}\label{subsec:enumcycles}
%%%%%%%%%%%%%%%%%%%%%%%%%%%%%%%%%%%%%%%%%%%%%%%%%%%%%%%%%%

Let $G =(V,E,v_G)$ be a directed graph with a weight function $v_G$ on $E$. The weight $v_G(\gamma)$ of a path $\gamma$ is the product of the weights of its arcs. A cycle of $G$ is a path ending at its starting point, up to a cyclic permutation. A path is {\em self-avoiding} if it does not visit the same vertex twice. A (non-empty) self-avoiding cycle is called an \emph{elementary cycle}. Two paths are disjoint if their vertex sets are disjoint, otherwise they are said to intersect.
\smallskip

Now consider the following model of heaps $\cH(G)$ attached to $G$: the basic pieces are the elementary cycles in $G$, and two such cycles $\gamma_1$ and $\gamma_2$ are in concurrence if they intersect. The weight (still denoted $v_G$) of an elementary cycle is the product of the edges it contains.

\begin{Theorem}[\cite{Viennot1}]\label{thm:bijHeapsCycles}
Let $u,v$ be two vertices in $G$. There is a weight-preserving bijection $\psi$ between 
\begin{itemize}
\item[(i)] the set of paths from $u$ to $v$ in $G$,  and 
\item[(ii)] the set of pairs $(\eta,H)$, where $\eta$ is a self-avoiding path from $u$ to $v$, and $H$ is a heap in $\cH(G)$ such that any maximal piece of $H$ intersects $\eta$.
\end{itemize}
\end{Theorem}
The bijection $\xi\mapsto (\eta,H)$ is obtained recursively on the length of the path as follows: if the path $\xi$ is a single vertex $u(=v)$, then $\eta$ is this trivial path and $H$ is empty. Now assume that the last arc of $\xi$ is $(v',v)$ and let $\xi'$ be the path from $u$ to $v'$ obtained by removing this arc. By induction, the bijection associates to $\xi'$ a pair $(\eta',H')$. Let $\eta_0$ be the concatenation of $\eta'$ and $(v',v)$. If $\eta_0$ is self-avoiding, define $\eta=\eta_0$ and $H=H'$. Otherwise,  $\eta_0$ decomposes uniquely into a self-avoiding path from $u$ to $v$, defined as $\eta$, and an elementary cycle $\gamma$ intersecting $\eta$ only in $v$. Define then $H=H'*\gamma$.

\begin{Example}%\rm
Consider the graph $G$ depicted in Figure~\ref{fig:Graphe_Cycle}, with all edge weights equal to $1$. There are four different cycles in $G$, precisely $\gamma_1=(BFCG)$, $\gamma_2=(BFA)$, $\gamma_3=(BCG)$, and $\gamma_4=(CED)$. The image of the path $\xi=ABFCGBFABCEDCE$ from $A$ to $E$ via the bijection of Theorem~\ref{thm:bijHeapsCycles} is $(\eta, H)$, where $\eta$ is the self-avoiding path $ABCE$, and the heap $H$ is obtained by the composition of $\gamma_1$ with $\gamma_2$, and $\gamma_4$.
\begin{figure}[!ht]
\begin{center}
 \includegraphics[width=0.85\textwidth]{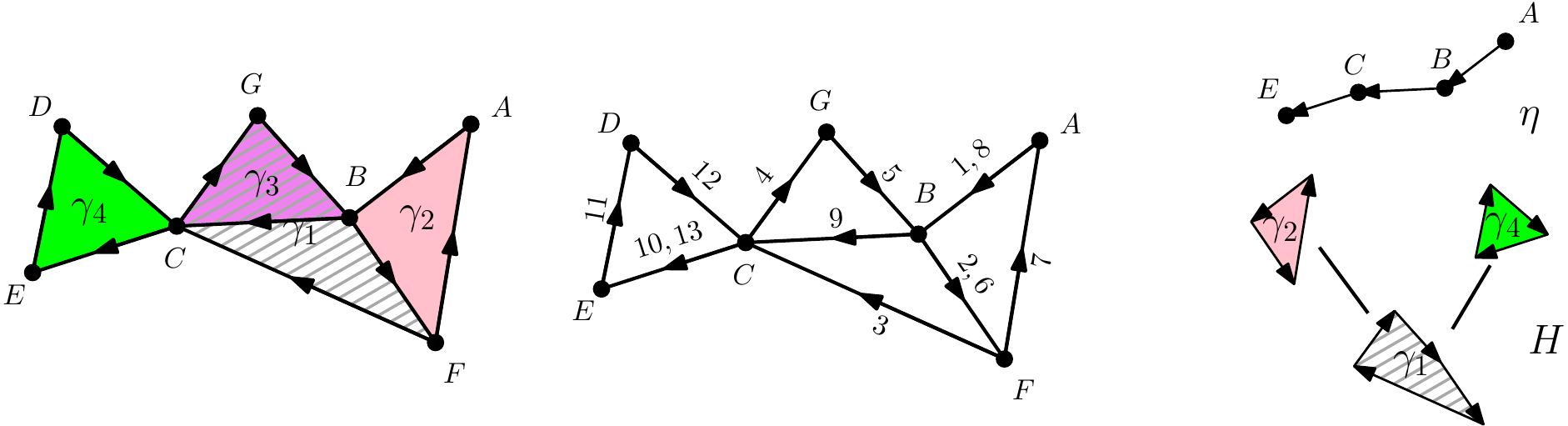}
 \caption{A directed graph $G$ (left), a path in $G$, represented by the labeling of its steps (center), and its image via $\psi$ (right).
 \label{fig:Graphe_Cycle}}
 \end{center} 
\end{figure}
\end{Example}

When $u=v$ in Theorem~\ref{thm:bijHeapsCycles}, then $\eta$ is necessarily the empty path $u$. This implies that $H$ can have only one maximal piece.

\begin{Corollary}\label{Cor:bijCyclesPyramids}
Let $u$ be a vertex in $G$. Then there is a weight-preserving bijection between paths from $u$ to itself and pyramids whose maximal piece contains the vertex $u$. 
\end{Corollary}

We will apply this result in Section~\ref{sec:heaps-monomeres-dimeres} to a graph in which the closed paths encode alternating diagrams.

%%%%%%%%%%%%%%%%%%%%%%%%%%%%%%%%%%%%%%%%%%%%%%%%%%%%%%%%%%%%%%
\section{Heaps of monomers and dimers}\label{sec:heaps-monomeres-dimeres}
%%%%%%%%%%%%%%%%%%%%%%%%%%%%%%%%%%%%%%%%%%%%%%%%%%%%%%%%%%%%%%

In this section, we give our first bijective approach regarding the enumeration of affine alternating diagrams. As will be shown, the case of finite alternating diagrams can be derived from the analysis of the affine case, therefore we will focus on the latter. 

The strategy is as follows: we will first translate bijectively  affine alternating diagrams in terms of a set $\Cylset^*$ of paths on a linear graph. As will be explained, this is a reformulation of a result in~\cite{BJN-long}. We will then be able to use Viennot's Theorem~\ref{thm:bijHeapsCycles} (actually Corollary~\ref{Cor:bijCyclesPyramids}) to translate the latter paths in terms of marked pyramids of monomers and dimers. By the Inversion Lemma, our enumeration problem will finally boil down to finding the generating functions for trivial heaps of monomers and dimers satisfying some specific conditions. 

The counterpart for involutions  will also be treated by this approach, as $321$-avoiding  affine involutions correspond to an explicit subset of paths in $\Cylset^*$, see~\cite{BJN-inv}. 

%%%%%%%%%%%%%%%%%%%%%%%%%%%%%%%%%%%%%%%%%%%%%%%%%%%%%%%
\subsection{From affine alternating diagrams to marked pyramids}
%%%%%%%%%%%%%%%%%%%%%%%%%%%%%%%%%%%%%%%%%%%%%%%%%%%%%%%

In~\cite{BJN-long}, affine alternating diagrams are put into correspondence with a set $\Cylset^*$ of lattice walks. We describe this set here in a different, though equivalent manner, keeping the same notation for simplicity. Consider the infinite graph $G$ depicted in  Figure~\ref{fig:Graphe_Motzbic}: vertices are labeled by nonnegative integers and edges are either loops labeled $L$ or $R$ (except at vertex $0$, where the only label for a loop is $L$), or directed edges $i\to i+1,\,i+1\to i$ for $i\geq0$. 

\begin{Definition}
We denote by $\Cylset^*$ the set of paths on $G$ which have the same starting and ending point. We also let $\Cylset_n^*$ be the set of paths of  $\Cylset^*$ of length $n$. For $\omega=j_0\to j_1\to \dots\to j_n=j_0 \in \Cylset^*_n$, we define ${\area}(\omega)=\sum_{i=0}^{n-1} j_i$.
\end{Definition}

Note that if we set $\area(e)=i$ for an edge starting at vertex $i$, then we have ${\area}(\omega)=\sum_{i=0}^{n-1} \area(j_i\to j_{i+1})$. It is then natural to define the weight $v_G$ on edges by $v_G(e)=xq^{\area(e)}$.

\begin{figure}[!ht]
\begin{center}
\includegraphics[width=8 cm]{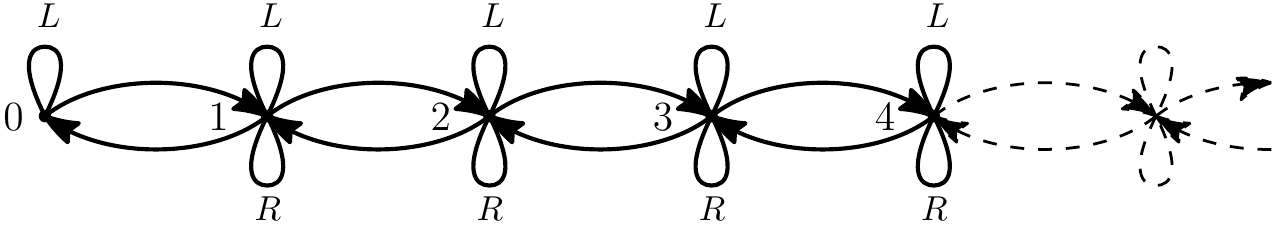}
\caption{Graph $G$ encoding walks in $\Cylset^*$.}
\label{fig:Graphe_Motzbic}  
\end{center}
\end{figure}

Now we associate any affine alternating diagram with a path in $\Cylset^*$. The correspondence goes as follows (see Figure~\ref{fig:walksPyramids}, left and middle, for an illustration of this bijection). For $n\geq2$, pick an affine alternating diagram $D \in \widetilde{\mathcal{D}}(n)$, and denote by $D_i$ the chain made of the elements labeled $s_i$. Since for any $i$, the chain $D_{i,i+1}$ is alternating, we have $|D_i|-|D_{i+1}| \in \{-1,0,1\}$. If $|D_i|=|D_{i+1}|>0$ the chain $D_{i,i+1}$ can be of two types: either $s_is_{i+1}\ldots s_is_{i+1}$ (called {\em type $R$}), or $s_{i+1}s_i\ldots s_{i+1}s_i$ (called {\em type $L$}). 

We then define $\varphi(D)$ as the path  $$|D_0| \rightarrow |D_1| \rightarrow |D_2| \rightarrow \cdots \rightarrow |D_n|=|D_0|$$ on the graph $G$, where if $|D_i|=|D_{i+1}|>0$, the loop $|D_i|\rightarrow |D_{i+1}|$ is the one with label the type of the chain $D_{i,i+1}$.

We let $\mathcal{E}_n$ be the subset of $\Cylset_n^*$ made of paths remaining at a fixed vertex $i>0$, and consisting of $n$ loops with identical label $L$, or $R$. The following result is a reformulation of~\cite[Theorem~2.2]{BJN-long} and \cite[Proposition~3.2]{BJN-inv}.

\begin{Theorem}\label{theorem:walksonG}
For $n\geq2$, the map $\varphi : \widetilde{\mathcal{D}}(n) \to \Cylset_n^*\setminus \mathcal{E}_n$ is a bijection such that $|D|=\area(\varphi(D))$. Moreover,
\begin{enumerate}
\item[$(i)$] $D$ is a finite alternating diagram if, and only if $\varphi(D)$ starts at vertex 0;
\item[$(ii)$] $D$ is self-dual if, and only if the only possible loops in the path $\varphi(D)$ are at vertex 0.
\end{enumerate}
%If $p=\varphi(D)$, then $|D|$ is equal to the sum of the values of the vertices of $p$ (by congiven by the sum of  the weights of all its loops, where the weight of a loop $i\to i$ is $i$, and the weight of a loop $i\to i+1\to i$ is $2i+1$, for any $i\geq0$.
\end{Theorem}

We can now connect affine alternating diagrams with a particular family of heaps by using Theorem~\ref{thm:bijHeapsCycles} and Corollary~\ref{Cor:bijCyclesPyramids} in Section~\ref{subsec:enumcycles}. 

Let $\cH(\cP_{md},\cC)$ be the model of heaps where the basic pieces are monomers $[i]$ with two possible labels $L$ and $R$, and dimers $[i,i+1]$, for nonnegative integers $i$. Let $\cH(\cP_{md}^*,\cC)$ be the same model but where the monomer $[0]$ occurs only with label $L$. Denote by $\Pi_{md}$ and $\Pi_{md}^*$ the sets of pyramids corresponding to  these two models. Finally, let $\Pi_{md}^{*\bullet}$ be the set of \emph{marked} pyramids in $\Pi_{md}^*$, \emph{i.e.} the set of pairs $(H,i)$ where $H$ is a pyramid with unique maximal segment $M$, and $i$ is the abscissa of one of the points of $M$.

\begin{Definition}\label{def:poidsmd}
To any heap $H$ in $\cH(\cP_{md},\cC)$ we associate a weight $v(H)$ as in~\eqref{eq:weight}, by assigning to monomers and dimers the respective  weights
\begin{equation}\label{eq:poidsMD}
v([i])=xq^i, \quad \mbox{and} \quad v([i;i+1])=x^2q^{2i+1}.
\end{equation}
More explicitly, in $v(H)$, the variable $x$ counts the number of monomers plus twice the number of dimers (equivalently, $\ell(H)+|H|$), while $q$ counts the sum of the abscissas of all the extremities of the segments of $H$. 
\end{Definition}

We now come back to the graph $G$.  Elementary cycles are either labeled loops $(i\to i)$  or cycles $(i\to i+1 \to i)=(i+1\to i \to i+1)$. By identifying $(i\to i)$ with the monomer $[i]$ and  $(i\to i+1 \to i)$ with the dimer $[i,i+1]$, the model of heaps $\cH(G)$ is identified with $\cH(\cP_{md}^ *,\cC)$. Moreover, $v_G((i\to i))=xq^i$ and $v_G((i\to i+1 \to i))=xq^ixq^{i+1}=x^2q^{2i+1}$, therefore $v_G$ coincides with the weight $v$ defined by~\eqref{eq:poidsMD} via this identification. Recall the bijection $\psi$  defined in Section~\ref{subsec:enumcycles}. Then  Corollary~\ref{Cor:bijCyclesPyramids} implies the following result.

\begin{Proposition}%\label{prop:affinepermutations}
The map $\bijheapscycle$ is a bijection between $\Cylset^*$ and $\Pi_{md}^{*\bullet}$ such that  if $\pi=\bijheapscycle(\omega)$ and $\omega$ has length $n$ and area $a$, then $v(\pi)=x^nq^a$. 
\end{Proposition}

An example is provided in Figure~\ref{fig:walksPyramids} , middle and right.

%\begin{equation}\label{eq:fgOMP}
%\sum_{\omega\in\Cylset^*}x^{\mbox{length}(\omega)}q^{\mbox{area}(\omega)}=\sum_{\pi\in\Pi_{md}^{*\bullet}}v(\pi).
%\end{equation}

\begin{figure}[!ht]
\begin{center}
 \includegraphics[width=\textwidth]{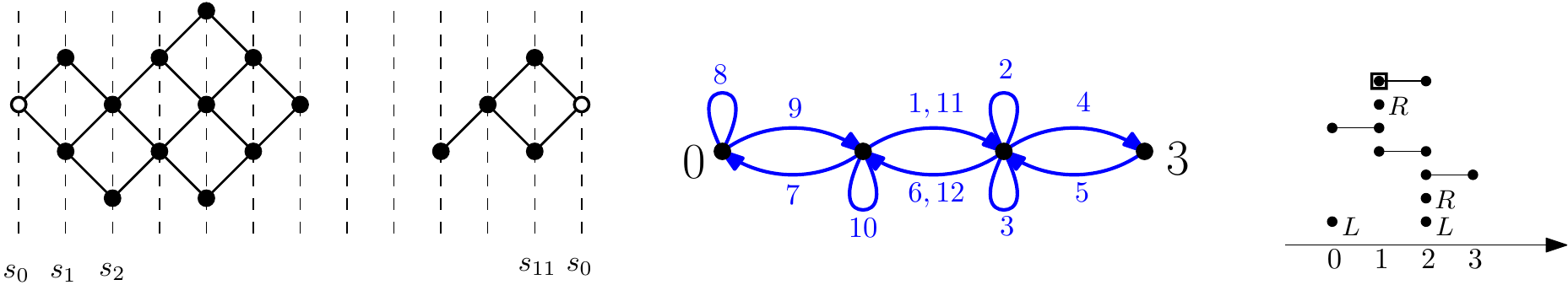}
 \caption{Bijection between affine alternating diagrams, paths on the graph $G$, and marked pyramids of monomers and dimers. The weight is $x^{12}q^{17}$.}\label{fig:walksPyramids}
 \end{center} 
\end{figure}

\begin{Remark}\label{rem:sansetoile}
Note that in the same way, the more general set $\Cylset$ of walks on the graph $G'$, defined by adding a loop labeled $R$ at vertex $0$ to the graph $G$, is in bijection with $\Pi_{md}^\bullet$, the set of marked pyramids in $\Pi_{md}$ without the additional condition on the monomers at abscissa $0$. 
\end{Remark}

%Finally, we let $\widetilde{\mathcal{D}}(0)$ and $\widetilde{\mathcal{D}}(1)$ be sets containing only the empty element.

Let $\Col_n$ be the image $\bijheapscycle(\mathcal{E}_n)$: it consists of heaps made of $n$ monomers at the same positive abscissa and all labeled $L$, or $R$. By combining the previous proposition with Theorem~\ref{theorem:walksonG}, we get the main result of this section.

\begin{Theorem}\label{thm:diagramstopyramids}
Let $n\geq 2$. The map $\BijDiagramPyramid:=\bijheapscycle\circ \varphi$  is a bijection between $\widetilde{\mathcal{D}}(n)$ and the set of pyramids $\pi$ in $\Pi_{md}^{*\bullet}\setminus\Col_n$ whose weight $v(\pi)$ has exponent $n$ in $x$. If $\pi=\BijDiagramPyramid(D)$, then $v(\pi)=x^nq^{|D|}$. Moreover,
\begin{enumerate}
\item[$(i)$] $D$ is a finite alternating diagram if, and only if the maximal piece of  $\pi$ is marked at vertex $0$; in particular, this piece must be of the form $[0]$ or $[0,1]$.
\item[$(ii)$] $D$ is self-dual if, and only if in $\pi$, the monomers may occur only at abscissa $0$.
\end{enumerate}

%\begin{equation}\label{eq:fgOMP}
%\sum_{n\geq0}\sum_{D\in\widetilde{\mathcal{D}}(n)}x^{n}q^{|D|}=\sum_{\pi\in\Pi_{md}^{*\bullet}}x^{\ell(\pi)+|\pi|}q^{h(\pi)}-2\sum_{n\geq1}\frac{x^nq^n}{1-q^n},
%\end{equation}
%and the second term of the right-hand side is avoided in cases $(1)$ and $(2)$.
\end{Theorem}

\subsection{Generating functions for 321-avoiding (affine) permutations}
%%%%%%%%%%%%%%%%%%%%%%%%%%%%%%%%%%%%%%%%%%%%%%%%%%%%%%%

In this subsection we provide  bijective proofs for the generating functions of 321-avoiding (affine) permutations and involutions given in Theorems~\ref{thm:A-tA} and~\ref{thm:A-tA-inv}. The method relies on the previous bijections, (some refinement of) the Inversion Lemma, and the computation of \emph{signed generating series}  $\sum_{T\in\cT}(-1)^{|T|}v(T)$ for specific sets $\cT$ of trivial heaps. 
More precisely, we will express our enumerative results in terms of the following series:

\beq\label{h-def}
\hh(x):= \sum_{n \ge 0} \frac{(-x)^n q^{\binom{n}{2}}(xq^n;q)_\infty}{(q;q)_n},
%\qquad \qquad 
%\HH(x)= \frac{h(x)}{(x)_\infty }
%=\sum_{n \ge 0} \frac{(-x)^n q^{n\choose 2}}{(q)_n(x)_n},
\eeq
and
\beq\label{j-def}
\jj(x):=\sum_{n\ge 0} \frac{(-x)^n q^{\binom{n}{2}}(xq ^{n+1};q)_\infty}{(q;q)_n},
\eeq
where for $n\ge 0$, we  recall the $q$-Pochhammer symbol $(x;q)_n= (1-x)(1-xq)\cdots(1-xq^{n-1})$, whose definition is extended to the limit case $n=\infty$ as an infinite product. By using the expressions~\eqref{h-def} and~\eqref{j-def} we first note that
\beq\label{jh}
\jj(x)=\hh(x)+x\hh(xq).
\eeq
The next result is a crucial tool towards our enumeration purposes. We postpone its combinatorial proof to the next subsection.
\begin{Theorem}\label{theorem:trivialh}
The signed generating series of the set $\cT_{md}$ of trivial heaps  of monomers (labeled $L$ or $R$) and dimers, is equal to $h(x)$. 
\end{Theorem}

We actually need the signed generating function for the set of trivial heaps with pieces in $\cP_{md}^*$. Starting from a trivial heap in $\cT_{md}$, decompose it according to whether it contains a monomer labeled $R$ at abscissa $0$ or not. If there is such a monomer, then by Theorem~\ref{theorem:trivialh} the corresponding signed generating series is $-xh(xq)$. Therefore $h(x)+xh(xq)$ is the signed generating series of the set $\cT_{md}^*$, and~\eqref{jh} has the following consequence.
 
\begin{Corollary}\label{coro:trivialj}
The  signed \gf \  of  the set $\cT_{md}^*$ of trivial heaps of monomers (labeled $L$ or $R$, except at abscissa $0$ where the only label is $L$) and dimers, is equal to $\jj(x)$.
\end{Corollary}

 We are now ready to derive combinatorially Theorem~\ref{thm:A-tA}. 

\begin{proof}[Proof of Theorem~\ref{thm:A-tA}.]
Let us define $\Pi_{md}^{*\bullet}(x;q):=\sum_{\pi\in\Pi_{md}^{*\bullet}}v(\pi)$, where $v(\pi)$ is the weight of Definition~\ref{def:poidsmd}. By Theorems~\ref{thm:bijectionSD} and~\ref{thm:diagramstopyramids}, we have
\beq\label{tA-P}
\tS(x,q)= \Pi_{md}^{*\bullet}(x;q) -2\sum_{n\ge 1} \frac{x^n q^n}{1-q^n},
\eeq
since the second term on the right-hand side enumerates heaps in the sets $\Col_n$, for $n\geq1$. We therefore need to compute the generating function of the set $\Pi_{md}^{*\bullet}$ of pyramids in $\Pi_{md}^{*}$ which are marked on their maximal segment. This is a consequence of the Inversion Lemma~\ref{lem:inversion}: indeed, one only needs to refine the proof of Corollary~\ref{Cor:enumpyramids} by taking into account the mark on one point of the maximal piece $M$. This gives
\[
\Pi_{md}^{*\bullet}(x;q)=\sum_{M\in\cP_{md}^{*}}(\ell(M)+1)\left(\sum_{H:\, \Max(H)\subseteq \{M\}}v(H)-1\right),
\]
where $\ell(M)$ is the length of the piece $M$. By applying Lemma~\ref{lem:inversion} and exchanging the summations we get
$$\sum_{\pi \in \Pi_{md}^{*\bullet}} v(\pi)=\displaystyle{-\frac{\displaystyle \sum_{T \in \cT_{md}^*} (-1)^{|T|}(\ell(T)+|T|)v(T)}{\displaystyle\sum_{T\in \cT_{md}^*} (-1)^{|T|}v(T)}},
$$
where $\ell(T)$ is the length of $T$. By Corollary~\ref{coro:trivialj}, the denominator above is $\jj(x)$ defined in~\eqref{j-def}, and therefore we get
$$
\Pi_{md}^{*\bullet}(x;q)=- x\frac{j'(x)}{\jj(x)},
$$
where the derivative is taken with respect to $x$.
Next, by using the definitions~\eqref{J-def} and~\eqref{j-def} of $J$ and $j$, we have $J(x)=j(x)/(xq;q)_\infty$, from which we deduce
$$
\Pi_{md}^{*\bullet}(x;q)=- x\frac{J'(x)}{J(x)}-x\sum_{i\ge 1} \frac{-q^i}{1-xq^i}=- x\frac{J'(x)}{J(x)}+\sum_{n\ge 1} \frac{x^n q^n}{1-q^n}.
$$
Returning to~\eqref{tA-P},  this gives the second result of
Theorem~\ref{thm:A-tA}.

We now consider $321$-avoiding permutations, which are by Theorem~\ref{thm:diagramstopyramids}~(i) in bijection with elements of $\Pi_{md}^{*\bullet}$ whose maximal piece has the form $[0]$ or $[0,1]$ and is marked at vertex $0$. This implies that the mark gives no information and can therefore be forgotten. By 
the Inversion Lemma~\ref{lem:inversion}, one gets:
\beq\label{A-inv}
xS(x,q)= \frac{\hh_0(x)}{\jj(x)}-1,
\eeq
where  $\hh_0(x)$ is the signed \gf\ of trivial heaps in $\cT_{md}$ that have no monomer or dimer at abscissa $0$. Since such heaps are obtained by translating one step to the right any trivial heap in $\cT_{md}$, we have $\hh_0(x)=\hh(xq)$, as $h$ is by Theorem~\ref{theorem:trivialh} the signed generating function of trivial heaps  in $\cT_{md}$. Combining this and~\eqref{A-inv}, we get by the definitions~\eqref{h-def} and~\eqref{j-def}:
$$S(x,q)=\frac{h(xq)-\jj(x)}{x\jj(x)}=\frac{\jj(xq)}{\jj(x)}.$$
This proves the first result of Theorem~\ref{thm:A-tA} by using the above relation between $J$ and $j$.
\end{proof}

In the same spirit, we prove combinatorially the counterpart for involutions.

\begin{proof}[Proof of Theorem~\ref{thm:A-tA-inv}.]
Thanks to Theorem~\ref{thm:diagramstopyramids}~(ii), the marked pyramids that we have to enumerate may have monomers only at abscissa $0$, with label $L$. Note that this
automatically rules out pyramids consisting of monomers lying at
positive abscissa. Using the same argument as in the previous proof, we obtain
$$
\ctS(x,q)= - x\frac{\cj'(x)}{\cj(x)},
$$
where $\cj(x)$ is the signed generating series for the set $\cT_L$ of trivial heaps of dimers and eventual monomers labeled $L$ lying at abscissa $0$. 
First note that the signed generating series for the set of trivial dimers is given by
\begin{equation}\label{part}
\ch(x):=\sum_{n\geq0}\frac{(-1)^{n} x^{2n}q^{2n \choose 2}}
{(q^2;q^2)_{n}}.
\end{equation}
Indeed, each such trivial heap with $n$ dimers corresponds to an integer partition $\lambda$ having $n$ odd parts such that the difference between two consecutive parts is greater or equal to $4$: to each dimer $[i,i+1]$, it suffices to associate the part $2i+1$ of $\lambda$. It is then classical to prove that the generating function $\sum_\la(-x^2)^{\ell(\lambda)}q^{|\lambda|}$, where $\ell(\lambda)$ is the number of parts, is equal to~\eqref{part}.

Next, by discussing whether there is a monomer at abscissa $0$ or not, the signed generating series of the set $\cT_L$ can be computed in a direct way as $$\cj(x)=\sum_{n\geq0}\frac{(-1)^{n} x^{2n}q^{2n \choose 2}}
{(q^2;q^2)_{n}}-x\sum_{n\geq0}\frac{(-1)^{n} (xq)^{2n}q^{2n \choose 2}}
{(q^2;q^2)_{n}}=\UU(x),$$
where $\UU(x)$ is defined in~\eqref{DD-def}. This gives the second result of Theorem~\ref{thm:A-tA-inv}. 

Finally, when we restrict our study to $321$-avoiding involutions, Theorem~\ref{thm:diagramstopyramids} yields a bijection with pyramids whose pieces are in $\cP_{md}$, with unique maximal piece $[0]$ or $[0,1]$, and having eventual monomers only at abscissa $0$, labeled $L$. As for~\eqref{A-inv}, the \gf\ now reads 
$$x\cS(x,q)= \frac{\ch_0(x)}{\UU(x)}-1,
$$
where  $\ch_0(x)$ is the signed \gf\ of trivial heaps of dimers at positive abscissa. As we have 
$$\UU(x)-\ch_0(x)=-x\UU(-xq),$$
this yields the first expression of Theorem~\ref{thm:A-tA-inv}.
\end{proof}

\begin{Remark}\label{rk:gfchemins}
\emph{One can also obtain closed form expressions for the
generating series of the sets of walks $\Cylset$ and $\Cylset^*$, denoted respectively $O(x)$ and $O^*(x)$ in~\cite{BJN-long}, using this approach. With the above notation,
$$
O(x)=-x\frac{h'(x)}{\hh(x)}\quad\mbox{and}\quad
O^*(x)=-x\frac{j'(x)}{\jj(x)},
$$
where the series $h$ and $j$ are defined in~\eqref{h-def} and~\eqref{j-def}, respectively. Similar expressions relate the generating series of walks, defined in~\cite{BJN-inv} and corresponding to involutions, to the series $\UU$ above and its companion $\ch$: if we denote by $\bar O$ (\emph{resp.} $\bar O^*$) the generating function for walks in $\Cylset^*$ having no loop $i\to i$ (\emph{resp.} only at vertex $0$), then we have
$$
\bar O(x)=-x\frac{\ch'(x)}{\ch(x)}\quad\mbox{and}\quad
\bar O^*(x)=-x\frac{\UU'(x)}{\UU(x)},
$$
where $\ch(x)$ is defined in~\eqref{part}.}
\end{Remark}

%%%%%%%%%%%%%%%%%%%%%%%%%%%%%%%%%
\subsection{Enumeration of trivial heaps of monomers and dimers}\label{subseq:enumeration}
%%%%%%%%%%%%%%%%%%%%%%%%%%%%%%%%%

We now give a combinatorial proof of Theorem~\ref{theorem:trivialh}. Our first step is to simplify the set $\cT_{md}$ that we have to enumerate. Consider the trivial heaps $T$ in $\cT_{md}$ that contain a dimer $[i,i+1]$ for a certain $i\geq 0$, or contain two monomers labeled $L$ and $R$ at positions $i$ and $i+1$ (or contain both configurations). On the set of such heaps, define the function $T\mapsto I(T)$ by first considering $i$ minimal such that one of the two cases occurs, and then by exchanging the dimer case with the consecutive monomer case, see an example below. This is clearly an involution,  which preserves weights since $x^2q^{2i+1}=(xq^i)(xq^{i+1})$, and switches the sign since the total number of pieces changes by $\pm 1$.
 
%\begin{figure}[!h]
\begin{center}
\includegraphics[width=12cm]{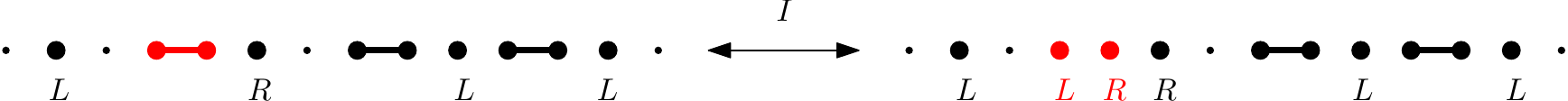}
%\caption{The killing involution $I$ on trivial heaps.\label{fig:invol_LR}}  
\end{center}
%\end{figure}

We thus have to enumerate trivial heaps of
monomers labeled $L$ or $R$, where no
monomer labeled $L$ in position $i$ can be followed by a monomer
labeled $R$ in position $i+1$. These heaps can be naturally considered as infinite words $w=w_0w_1w_2\cdots$ on the alphabet $\{0,L,R\}$ with a finite number of nonzero letters and which avoid the (contiguous) factor
$LR$. Let us call this set of words $\mathcal{W}$, and for any word $w\in\mathcal{W}$, denote by $|w|_L$ (\emph{resp.} $|w|_R$) the number of occurrences of $L$ (\emph{resp.} $R$) in $w$. Our task is now to show that $\hh(x)$ is the generating function of $\mathcal{W}$ with sign $(-1)^k$ and weight
 $x^kq^l$, where $k$ is the total number of nonzero letters and $l$ is the sum of their indices in the word. Summarizing, we need to show
 \begin{equation}
 \label{eq:hword}
 \hh(x)=\sum_{w\in\mathcal{W}}(-x)^{|w|_L+|w|_R}q^{\sum_{w_i\in\{L,R\}} i}.
 \end{equation}

Consider the application $Pr:w\to(\pi_R(w),\pi_L(w))$
where $\pi_R(w)$ (\emph{resp.} $\pi_L(w)$) is obtained from $w$ by
removing all occurrences of $L$ (\emph{resp.} occurrences of $R$) in $w$. 

\begin{Lemma}
\label{lemma}
$Pr$ is a bijection from $\mathcal{W}$ to the set
of pairs $(w^R,w^L)$ of words where $w^R$ (\emph{resp.} $w^L$) is a
word on $\{0,R\}$ (\emph{resp.} $\{0,L\}$) with a finite number of nonzero letters. 

Moreover for any word $w$ in $\mathcal{W}$,
\begin{equation}
\label{eq:weights}
\sum_{w_i\in\{L,R\}} i=\sum_{\pi_L(w)_i=L} i+\sum_{\pi_R(w)_i=R} i+|w|_R|w|_L.
\end{equation}
\end{Lemma}

\begin{proof}
 The inverse bijection goes as follows: given $(w^R,w^L)$, consider the occurrences $r_0:=-1< r_1<r_2<\cdots$ of the zeros in $w^R$ and
$l_0:=-1< l_1<l_2<\cdots $ of the zeros in $w^L$. Then the word $w$ is obtained by having the factor $R^{r_{i}-r_{i-1}-1}L^{l_{i}-l_{i-1}-1}$ between the $i-1$th and $i$th occurrence of $0$.  

Now introduce the notation $|u|_{ab}=|\{i<j\mid  u_i=a, u_j=b\}|$ for any word $u=u_0u_1u_2\cdots$ on an alphabet $A$ and any letters $a,b\in A$. Then one has 
\begin{align*}
\sum_{w_i\in\{L,R\}} i &=\sum_{w_i=L} i + \sum_{w_i=R} i\\
&=|w|_{0L}+|w|_{LL}+|w|_{RL}+|w|_{0R}+|w|_{RR}+|w|_{LR}\\
&=(|w|_{0L}+|w|_{LL})+(|w|_{0R}+|w|_{RR})+(|w|_{RL}+|w|_{LR})\\
&=\sum_{\pi_L(w)_i=L} i+\sum_{\pi_R(w)_i=R} i+|w|_R|w|_L,
\end{align*}
which is precisely~\eqref{eq:weights}.
\end{proof}

\begin{figure}[!h]
\begin{center}
\includegraphics[width=9cm]{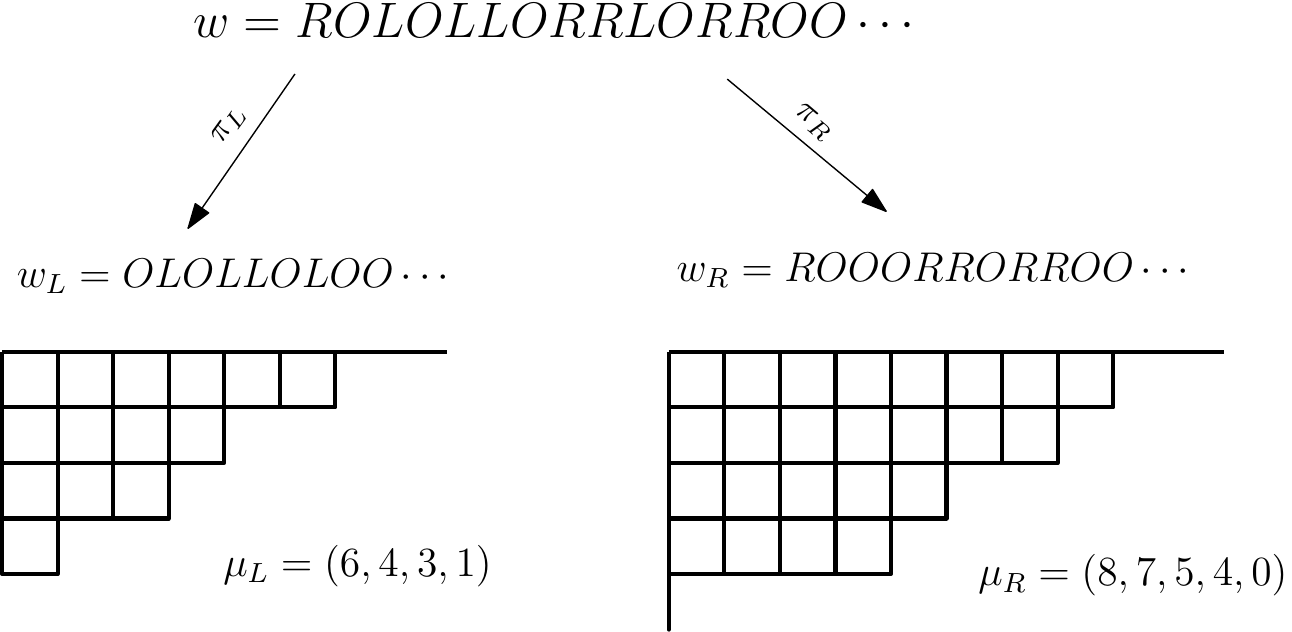}
\caption{Bijection between $\mathcal{W}$ and pairs of partitions with distinct parts.}\label{fig:bij_LR}  
\end{center}
\end{figure}

We can now finish the proof of Theorem~\ref{theorem:trivialh}. Given $w_R$ and $w_L$, encode them by two integer
partitions $\mu_R$ and $\mu_L$ with distinct parts in
$\{0,1,2,\dots\}$, by recording the positions
of the letters $R$ and $L$. See Figure~\ref{fig:bij_LR} for an example. Let $\mathcal{D}$ be this set of partitions. For $\mu\in \mathcal{D}$ we consider the weight $f(\mu)=(-x)^{\ell(\mu)}q^{|\mu|}$, where $\ell(\mu)$ is the number of parts of $\mu$ and  $|\mu|$ is the sum of its parts. The bijection $w \to (\mu_R,\mu_L)$ is then weight-preserving if we give the weight $f(\mu_R)f(\mu_L)q^{\ell(\mu_R)\ell(\mu_L)}$ to $(\mu_R,\mu_L)$, thanks to~\eqref{eq:weights} above.

To conclude the proof of~\eqref{eq:hword}, we must finally show that
\begin{equation}\label{eq:hpart}
\sum_{\lambda,\mu\in\mathcal{D}}(-x)^{\ell(\lambda)+\ell(\mu)}q^{|\lambda|+|\mu|+\ell(\lambda)\ell(\mu)}
\end{equation}
is equal to $\hh(x)$. For any partition $\lambda\in\mathcal{D}$ of length $n$, writing $\tilde{\lambda}:=(\lambda_1-(n-1),\dots,\lambda_{n-1}-1,\lambda_n-0)$,
then $\tilde{\lambda}$ is a partition with at most $n$ parts, such that $|\lambda|=|\tilde{\lambda}|+\bi{n}{2}$. Therefore~\eqref{eq:hpart} is equal to
$$
\sum_{n\geq0}\frac{(-x)^nq^{\bi{n}{2}}}{(q;q)_n}\sum_{\mu\in\mathcal{D}}(-xq^n)^{\ell(\mu)}q^{|\mu|}
=\sum_{n\geq0}\frac{(-x)^nq^{\bi{n}{2}}}{(q;q)_n}(xq^n;q)_\infty,
$$
which is the definition~\eqref{h-def} of $\hh(x)$.

%%%%%%%%%%%%%%%%%%%%%%%%%%%%%%%%%%%%%%%%%%%%%%%%%%%%%%%%%%%%%%%%%%%%%%%%%%%
\section{Periodic parallelogram polyominoes and heaps of segments}
\label{sec:parallelogram}
%%%%%%%%%%%%%%%%%%%%%%%%%%%%%%%%%%%%%%%%%%%%%%%%%%%%%%%%%%%%%%%%%%%%%%%%%%%

We introduce  \emph{periodic parallelogram polyominoes}, which are a natural extension of classical parallelogram polyominoes. We enumerate them  using heaps of segments, extending the approach from~\cite{mbm-viennot} which was applied to the case of the usual parallelogram polyominoes.  We then relate this setting to the $q$-enumeration of $321$-avoiding affine permutations, giving a second bijective proof of Theorem~\ref{thm:A-tA}.

In all this section, let $\H:=\H(\cP,\cC)$ be the set of heaps of segments introduced in Section~\ref{subsec:defiheaps}: $\cP$ is the set of segments $[a,b]$ with $a,b$ integers and $1\leq a\leq b$, and two segments are concurrent if they intersect.
 
%%%%%%%%%%%%%%%%%%%%%%%%%%%%%%%%%%%%%%%%%%%%%%%%%%%%%%%%%%%%%%%%%%%%%%%%%%%
\subsection{Heaps of segments and alternating sequences}
\label{subsec:heapsandsequences}
%%%%%%%%%%%%%%%%%%%%%%%%%%%%%%%%%%%%%%%%%%%%%%%%%%%%%%%%%%%%%%%%%%%%%%%%%%%
\subsubsection{Sequences}
 \begin{Definition}\label{def:sequence} For $n\geq 0$, let ${\mathbf S}_n$ be the set of sequences $(a_i,b_i)_{1\leq i\leq n}$ of pairs of integers satisfying $1\leq a_i\leq b_i$ for all $i$ and $a_i\leq b_{i-1}$ for $i>1$, i.e.
\begin{equation}
\label{eq:alternating}
a_1\leq b_1\geq a_2\leq b_2\geq \cdots\leq b_{n-1}\geq a_n\leq b_n.
\end{equation}
\end{Definition}
Define  ${\mathbf S}:=\cup_{n\geq 0} {\mathbf S}_n$, and consider $\seq=(a_i,b_i)_{1\leq i\leq n}\in{\mathbf S}$. Following~\cite[Section III]{mbm-viennot}, we associate with $\seq$ a heap $f(\seq)\in\H$ as
\[
f(\seq):=[a_n,b_n]*\dots* [a_1,b_1],
\]
where we recall from Section \ref{subsec:defiheaps} that $*$ denotes the composition of heaps. In words, $f(\seq)$ is  obtained by stacking the segments $[a_n,b_n]$, $[a_{n-1},b_{n-1}]$, \ldots, finishing by $[a_1,b_1]$. For instance the heap below is the image of the sequence $(a_i,b_i)_{1\leq i\leq 5}=(2,5),(5,7),(3,7),(1,2),(1,1)$.

\begin{center}
\includegraphics[scale=0.8]{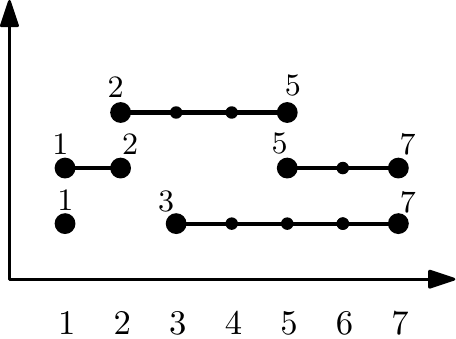}
\end{center}

\begin{Lemma}
Let $\seq=(a_i,b_i)_{1\leq i\leq n}\in{\mathbf S}_n$. Then $[a_n,b_n]$ is the leftmost minimal segment in $f(\seq)$, and $[a_1,b_1]$  is the rightmost maximal segment.
\end{Lemma}

\begin{proof}
 The result for $[a_n,b_n]$ is shown in~\cite[Lemme 3.3(i)]{mbm-viennot} in the special case where $a_1=1$, but the proof applies verbatim. 
 
 To prove the property for $[a_1,b_1]$, one can argue by symmetry as follows: let $a=\min_i a_i$ and $b= \max_i b_i$. Consider the sequence $\seq'=(a'_i,b'_i)_i$ with $a'_i=a+b-b_{n-i+1}$ and $b'_i=a+b-a_{n-i+1}$ for ${1\leq i\leq n}$. This sequence belongs to ${\mathbf S}_n$, and $\seq\mapsto \seq'$ is clearly an involution. By definition of $f$, the heap $H':=f(\seq')$ is obtained from $H:=f(\seq)$ by performing a half-turn. In particular the rightmost maximal segment of $H$ corresponds to the leftmost minimal segment of $H'$. But we know that the latter is $[a'_n,b'_n]=[a+b-b_1,a+b-a_1]$, which corresponds to  $[a_1,b_1]$ in $H$.
\end{proof}

This shows that $f$ is injective: indeed the inverse image of any heap of segments is uniquely determined by successively removing the leftmost minimal segments in a heap of $\H$ and  naming the first one $[a_n,b_n]$, the  second one $[a_{n-1},b_{n-1}]$, and so on until $[a_1,b_1]$.
 Now for any heap $H$ the sequence thus constructed will belong to ${\mathbf S}$ because taking the leftmost minima at all steps ensures the inequalities $a_i\leq b_{i-1}$ for $i>1$. This shows that $f$ is surjective and we get the following result.

\begin{Proposition}\label{prop:pairofsequences}
 $f$ is a bijection between ${\mathbf S}$ and $\H$.
\end{Proposition}

\subsubsection{Parallelogram Polyominoes} A \emph{parallelogram polyomino} (PP) is a subset of $\rs^2$, defined up to translation, by the region enclosed between two finite paths in $\zs^2$, using East and North steps with common endpoints but which do not otherwise intersect~\cite{mbm-viennot}. We can view such a polyomino as its sequence of columns $C_1,\ldots,C_n$ going from left to right, where $C_{i+1}$ has its bottom (\emph{resp.} top) cell, not lower than the bottom (\emph{resp.} top) cell of $C_i$. An example is given in the left of Figure~\ref{fig:polytoheap} (ignore the dotted lines).

 For $i=1,\ldots ,n$, let $b_i$ be the number of cells of $C_i$, and $a_i$ be the number of common rows between $C_{i-1}$ and $C_i$, where by convention $a_1=1$. As noticed in~\cite{mbm-viennot}, this encoding is a bijection from parallelogram polyominoes with $n$ columns to the subset of ${\mathbf S}_n$ with $a_1=1$. By the results of the previous subsection, $f$ thus induces a bijection between parallelogram polyominoes and \emph{semi-pyramids}: these are pyramids in $\H$ whose maximal piece has the form $[1,b]$. This is precisely~\cite[Proposition 3.4(i)]{mbm-viennot}.\smallskip

We need to slightly extend this correspondence. Define a \emph{pointed PP} as a pair $(P,{c})$ where $P$ is a PP and ${c}$ is a positive integer less than or equal to $b_1$, the height of the first column of $P$. By taking $a_1={c}$, it is then clear that pointed PPs are in bijection with $\mathbf{S}$. We can now define our main objects of study of the present section.

\begin{Definition}\label{def:ppp}
A \emph{periodic parallelogram polyomino} (PPP) is a pointed PP $(P,{c})$ in which ${c}$ is at most the height of $C_n$.
\end{Definition}

Let $\widetilde{\mathbf S}$ be the set of $(a_i,b_i)_{1\leq i\leq n}\in{\mathbf S}$ such that $a_1\leq b_n$. We immediately have the following result.

\begin{Proposition}\label{prop:PPP-S}
The set of PPPs is in bijection with $\widetilde{\mathbf S}$.
\end{Proposition}

 An example of a PPP is represented in Figure~\ref{fig:polytoheap}, together with its induced image under $f$. The dashed columns in the picture highlight the \emph{periodic} structure: the mark ${c}$ tells us how to ``glue'' the extreme columns of $P$.

\begin{figure}[!ht]
\includegraphics[width=0.6\textwidth]{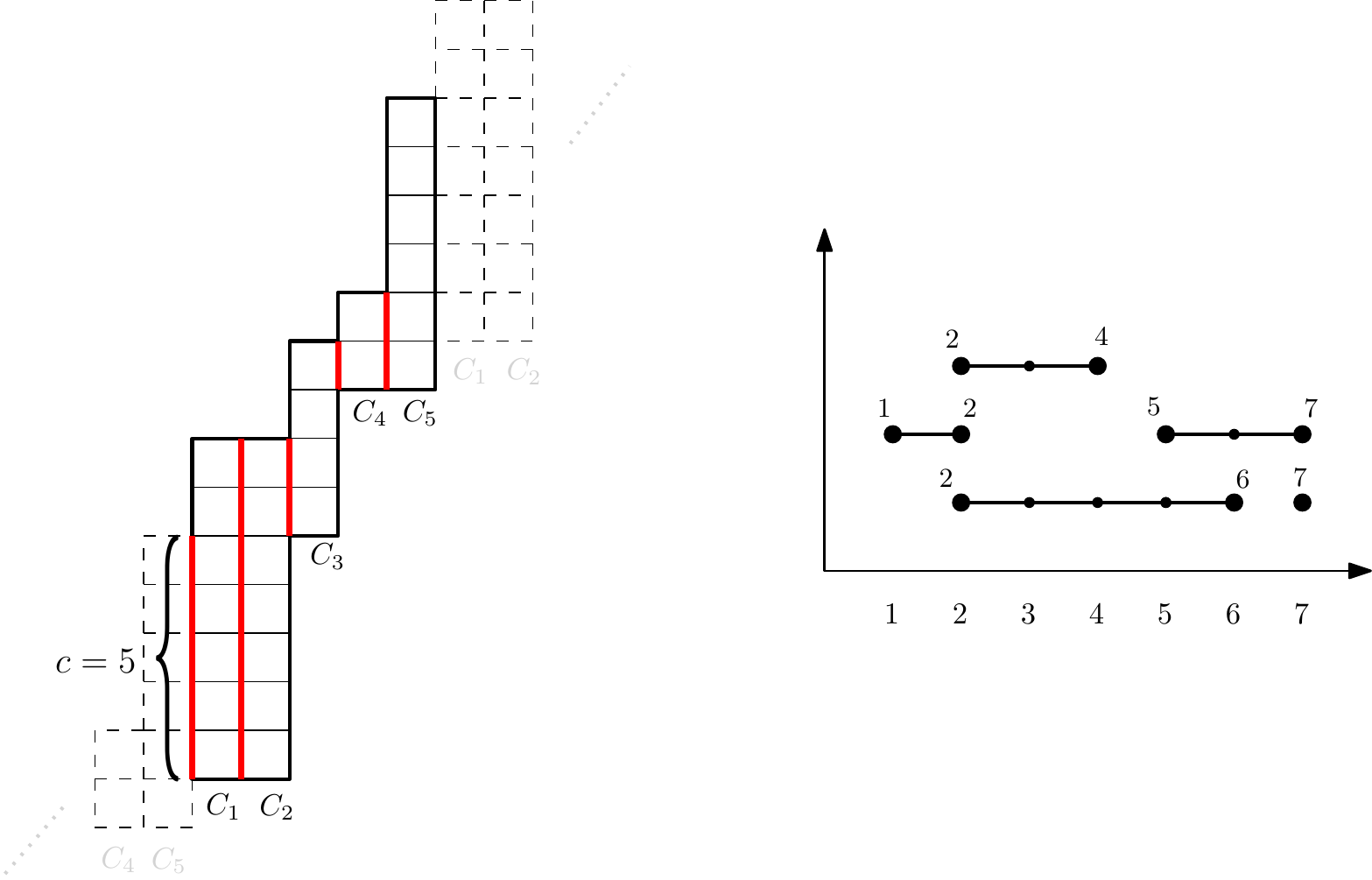}
\caption{A periodic parallelogram polyomino $(P,5)$ of width $n=5$ and its image $f(P)$, with $(a_i,b_i)_{1\leq i\leq 5}=(5,7),(7,7),(2,4),(1,2),(2,6)$.}
\label{fig:polytoheap}
\end{figure}

\begin{Definition}\label{def:Htilde} Let $\Ht\subset \H$ be the set of heaps $H$ such that either $H$ is the empty heap, or $a\leq b'$ where $[a,b]$ is the rightmost maximal segment of $H$ and $[a',b']$ is its leftmost minimal segment.
Equivalently, $H\in\Ht$ if there is no minimal segment of $H$ which occurs completely to the left of a maximal segment. 
\end{Definition}

From the preceding discussion we immediately derive the following result.

\begin{Proposition}\label{prop:bijcyclicheap}
The bijection $f$ restricts to a bijection between $\widetilde{\mathbf S}$ and $\Ht$, and thus induces a bijection between PPPs and $\Ht$.
\end{Proposition}

We have arrived at a description of PPPs in terms of heaps, which will be the base for our enumeration in the next subsection.

%%%%%%%%%%%%%%%%%
\subsection{Enumeration of PPPs}% and proving Theorem~\ref{thm:enumPPP}
%%%%%%%%%%%%%%%%%

\subsubsection{Weights} We follow three natural statistics on PPPs. Let $(P,{c})$ be such a polyomino, and $\seq=(a_i,b_i)_{1\leq i\leq n}\in \widetilde{\mathbf S}$ its associated sequence. Its \emph{width} is the number $n$ of columns of $P$, while its \emph{area} is the number of cells in $P$, which can be computed as $\sum_i b_i$. Its \emph{height} is the height of $P$ as a parallelogram polyomino \emph{minus }$a_1(=c)$, so that the height is given by $\sum_i(b_i-a_i)$. The PPP of Figure~\ref{fig:polytoheap} has width $5$, height $9$ and area $26$.  Our goal is to count PPPs according to the weight $w(P,c)=x^\text{height}y^\text{width}q^\text{area}$.
\smallskip

 Given a basic segment $s=[a,b]$, let its weight be defined as $v(s)=x^{b-a}yq^b$. For a heap $H\in \H$, the induced weight $v(H)=\prod_{s\in H}v(s)$ is then given by
\begin{equation}
\label{eq:heapweight}
v(H)=x^{\ell(H)}y^{|H|}q^{e(H)},
\end{equation}
where  $\ell(H)$ is the sum of the lengths of all segments, $|H|$ is the number of segments in $H$ and $e(H)$ is the sum of the values of all right endpoints of segments.

We have finally that if $H$ is the heap corresponding to $(P,{c})$ by Proposition~\ref{prop:bijcyclicheap}, then the weights match: $v(H)=w(P,{c})$. We have thus transformed the problem of enumerating PPPs to the problem of enumerating heaps in $\Ht$: let $PPP(x,y,q)$ and $PP(x,y,q)$ be the generating functions of PPPs and PPs respectively with respect to the weight $w$. Then we have 
\begin{equation}
\label{eq:polyominoes_as_heaps}
PPP(x,y,q)=\sum_{H\in\Ht}v(H)\quad\mbox{and}\;\;\;\; PP(x,y,q)=\sum_{H\in\SP}v(H),
\end{equation}
where $\SP$ is the set of semi-pyramids defined above.

\subsubsection{Results}  Denote by $\mathcal{T}\subset \H$ the set of trivial heaps and  by $\mathcal{T}_{>1}$ its subset of heaps containing no segment of the form $[1,b]$. Introduce their signed generating functions
  \[%\label{eq:trivialheaps}
 N(x,y,q):=\sum_{T\in \mathcal{T}} (-1)^{|T|}v(T)\quad\mbox{and}\;\;\;\; \hat{N}(x,y,q):=\sum_{T\in \mathcal{T}_{>1}} (-1)^{|T|}v(T).
  \]

The encoding of polyomino parallelograms as semi-pyramids, and the use of the Inversion Lemma~\ref{lem:inversion} imply:
\begin{equation}\label{eq:enumPP}
PP(x,y,q)=-x\frac{\hat{N}(x,y,q)}{N(x,y,q)}.
\end{equation}

Numerator and denomimator were computed bijectively in~\cite[Proposition 4.1]{mbm-viennot}:
 \begin{equation}
\label{eq:trivialheaps}
\hat{N}(x,y,q)=\sum_{n\geq 1}\frac{(-y)^nq^{\bi{n+1}{2}}}{(q;q)_{n-1}(xq;q)_n}\;\;\;\;\mbox{and}\;\;\;\;N(x,y,q)=\sum_{n\geq 0}\frac{(-y)^nq^{\bi{n+1}{2}}}{(q;q)_n(xq;q)_n}.
\end{equation}

We have for PPPs a result similar to~\eqref{eq:enumPP}.

\begin{Theorem}\label{thm:enumPPP}
The generating function $PPP(x,y,q)$ for periodic parallelogram polyominoes is given by
\[PPP(x,y,q)=-y\frac{\partial_y N(x,y,q)}{N(x,y,q)}.\]
\end{Theorem}

We prove this result in the rest of this section, before giving a bijective proof of Theorem~\ref{thm:A-tA} in Section~\ref{subsec:BackTo321}.

%%%%%%%%%%%%%%%%%
\subsubsection{Counting heaps in $\Ht$% and proving Theorem~\ref{thm:enumPPP}
}
%%%%%%%%%%%%%%%%%

The proof of Theorem~\ref{thm:enumPPP} relies on the Inversion Lemma, together with some technical combinatorial results.

Consider any nontrivial heap $F\in \H\setminus\mathcal T$. We need to define certain special pieces of $F$ to answer the following question in Lemma~\ref{lem:imagephi}: what are the possible subsets of minima that can be removed from $F$ to obtain a heap in $\Ht$ ? The reader is advised to take a look at Figure~\ref{fig:heaps_parameters} for an illustration of the definitions.

 Define $S_F=:[a_F,b_F]$ to be the rightmost segment in $\Max(F)\setminus \Min(F)$, and $Y_F$ as the set of segments $[a,b]\in \Max(F)$ which satisfy $a>b_F$. By definition of $S_F$, note that $Y_F\subset \Min(F)$. Let $X_F$ be the set of segments $[a,b]\in \Min(F)$ satisfying $b<a_F$.  We then define $U_1(F):=X_F\sqcup Y_F\subseteq \Min(F)$.

 Let $F'=F\setminus\Min(F)$. Suppose that $F'\notin\Ht$ and that there exists $S_0\in\Min(F)$ such that $S_0*F'\in\Ht$: in this case $F$ is said to be of \emph{type $0$}. Otherwise, it is said to be of \emph{type $1$}. Note that $S_0$ is unique if it exists, since its endpoints $[a_0,b_0]$ necessarily satisfy $a_0<a_F\leq b_0$ and at most one minimal element can satisfy these inequalities. If $F$ is of type $0$ (\emph{resp.} of type $1$) we define $U_2(F):=\Min(F)\setminus \{S_0\}$   (\emph{resp.}  $U_2(F):=\Min(F)$).

\begin{figure}[!ht]
\includegraphics[width=0.6\textwidth]{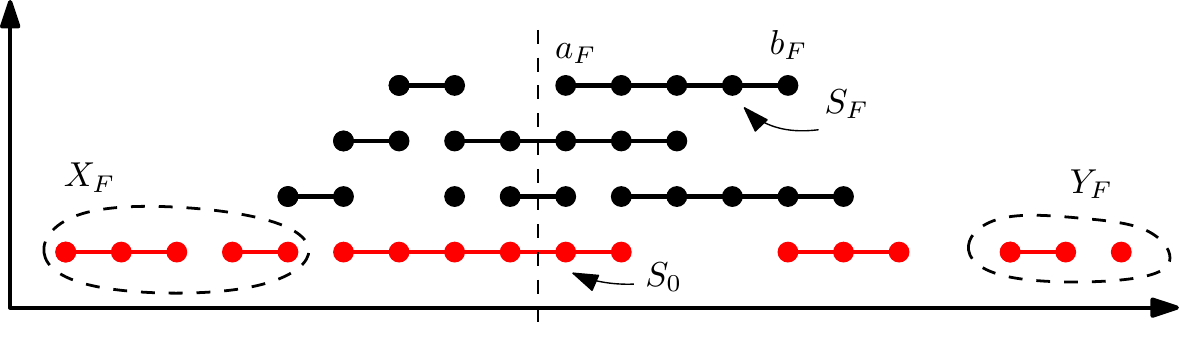}
\caption{Special pieces of a heap needed in Lemma~\ref{lem:imagephi}.}
\label{fig:heaps_parameters}
\end{figure}
We wish to apply the general principle of the Inversion Lemma~\ref{lem:inversion} in order to enumerate $\Ht$. Given any $T\in\mathcal{T}$ and $E\in\H$ we defined $\phi(T,E)=T*E$, (see Section~\ref{subsec:enumheaps}). We need to determine the image of $\phi$ when it is restricted to $\mathcal T\times\Ht$, and the antecedents of any element in this image.

\begin{Lemma}
\label{lem:imagephi}
 A heap $F$ belongs to $\phi( \mathcal T\times\Ht)$ if, and only if one of the following occurs:
\begin{enumerate}
\item  $F\in \mathcal T$. In this case if $U\subseteq \Min(F)$, then $F\setminus U\in\Ht$ iff $\left|F\setminus U\right|=1$.
\item  $F\not\in \mathcal T$ and $F\setminus U_1(F)\in\Ht$. In this case if $U\subseteq \Min(F)$, then $F\setminus U\in\Ht$ iff  $U_1(F)\subseteq U\subseteq U_2(F)$.
\end{enumerate} 
\end{Lemma}

\begin{proof}
Assume $F=T*E$ where $E\in \Ht$. 

 If $F$ is trivial, then a fortiori $E$ is trivial. Since $E$ is in $\Ht$ by hypothesis, the only possibility is that $E$ is reduced to a single segment, from which case (1) follows.

Now assume $F$ is non trivial. We first need to determine which pieces of $\Min(F)$ necessarily come from $T$. First, no piece in $Y_F$ can come from $E$: otherwise the rightmost maximum of $E$ is in $Y_F$, but then any non-minimal piece of $F$ contradicts the fact that $E$ is in $\Ht$. We infer immediately that \emph{$S_F$ is the rightmost maximum of $E$}. By the definition of $\Ht$, this implies that the minimal segments in $X_F$ cannot come from $E$ either. We have proved that, if we let $U$ be the image of $T$ in $F$, $U_1(F)\subseteq U$. 

In other words, $E$ is obtained from $F$ by removing a certain subset $U$ of $\Min(F)$ which contains $U_1(F)$. Note first that $F\setminus U_1(F)\in\Ht$. Then $E\in\Ht$ means that no minimum element of $F\setminus U$ occurs strictly left of $S_F$. Such a bad minimum could occur if, and only if $F$ has type $0$ and the segment $S_0$ is removed. This shows that one cannot have $S_0\in U$, which concludes the proof. 

\end{proof}

 We now follow the steps of proof of the Inversion Lemma~\ref{lem:inversion}:
 \begin{align*}
N(x,y,q)\, PPP(x,y,q)&=\left( \sum_{T\in \mathcal{T}} (-1)^{|T|}v(T)\right)\left(\sum_{E\in \Ht} v(E)\right)\\
&=\sum_{F\in\phi( \mathcal T\times\Ht)                      }v(F)\left(\sum_{U\subseteq\min(F),F\setminus U\in\Ht}(-1)^{|U|}\right)\\
&= \sum_{F\in\mathcal{T}}|F|(-1)^{|F|-1}v(F)+\sum_{F\notin \mathcal{T}, F\setminus U_1(F)\in\Ht}v(F)\left(\sum_{U_1(F)\subseteq U\subseteq U_2(F)}(-1)^{|U|}\right),
\end{align*}
where we used Lemma~\ref{lem:imagephi} in the last equality. The first term of the last expression is by inspection equal to $-y\partial_y N(x,y,q)$, so the proof of  Theorem~\ref{thm:enumPPP} will follow if one can prove that the second sum is equal to zero. Notice that if $F$ satisfies $U_1(F)\neq U_2(F)$ in this sum, then the inner sum is zero, so we can discard these heaps. We are thus led to consider the following set of heaps. 

\begin{Definition}
Let $\mathcal{W}$ be the set of nontrivial heaps such that $U_1(F)=U_2(F)$ and $F\setminus U_1(F)\in\Ht$.
\end{Definition}

 The preceding discussion shows that Theorem~\ref{thm:enumPPP} is a direct consequence of the following formula:
\begin{equation}
\label{eq:technical_vanishing}
\sum_{F\in\mathcal{W}}v(F)(-1)^{|U_1(F)|}=0.
\end{equation}

%We will prove this technical proposition in the next paragraph.

%%%%%%%%%%%%%%%%%
\subsubsection{Proof of Equation~\eqref{eq:technical_vanishing}}
%%%%%%%%%%%%%%%%%

  Let $\mathcal{W}_0$ and $\mathcal{W}_1$ denote the heaps of $\mathcal{W}$ of type $0$ and $1$ respectively. Notice that $U_1(F)=\Min(F)\setminus \{S_0(F)\}$ if $F\in\mathcal{W}_0$, while $U_1(F)=\Min(F)$ if $F\in\mathcal{W}_1$.

We have the following technical lemma.

\begin{Lemma}
\label{lem:CharacterizationW}
Let $F\in \mathcal{W}$, and $F':=F\setminus \Min(F)$.

If $F\in \mathcal{W}_0$, then all the minima of $F'$ are concurrent with $S_0(F)$. We denote by $S'_0(F)$ the leftmost such minimum. 

If $F\in \mathcal{W}_1$, then $F'$ has a unique minimum denoted by $S'_1(F)=:[a'_1,b'_1]$. If $S_1(F):=[a_1,b_1]$ is the rightmost minimum in $X_F$, then $a'_1\leq b_1$ and $a_1\leq b'_1$.
\end{Lemma}

\begin{proof}
Suppose $F\in \mathcal{W}_0$, so that $\Min(F)=X_F\sqcup Y_F\sqcup\{S_0(F)\}$. If a minimum of $F'$ was left of $S_0(F)$ then this would contradict $\{S_0(F)\}*F'\in\Ht$; if a minimum of $F'$ was on the right of $S_0(F)$ then it would be concurrent with $Y_F$ which contradicts the definition of $Y_F$. 

Suppose $F\in \mathcal{W}_1$, so that $\Min(F)=X_F\sqcup Y_F$. Let $[a'_1,b'_1]$ be an element of $\Min(F')$. With the same reasoning as above, one shows that $a'_1\leq b_1$, and $a_F\leq b'_1$ and the uniqueness of $S'_1(F)$ follows.  Since $a_1\leq a_F$, the wanted equalities are also proven.
\end{proof}

We can now define two functions $\psi_0$ and $\psi_1$ on $\mathcal{W}$. These are constructed by selecting two special segments in a heap and then exchanging their right endpoints, see Figure~\ref{fig:involution_psi} for an illustration. 

\begin{figure}[!ht]
\includegraphics[width=0.8\textwidth]{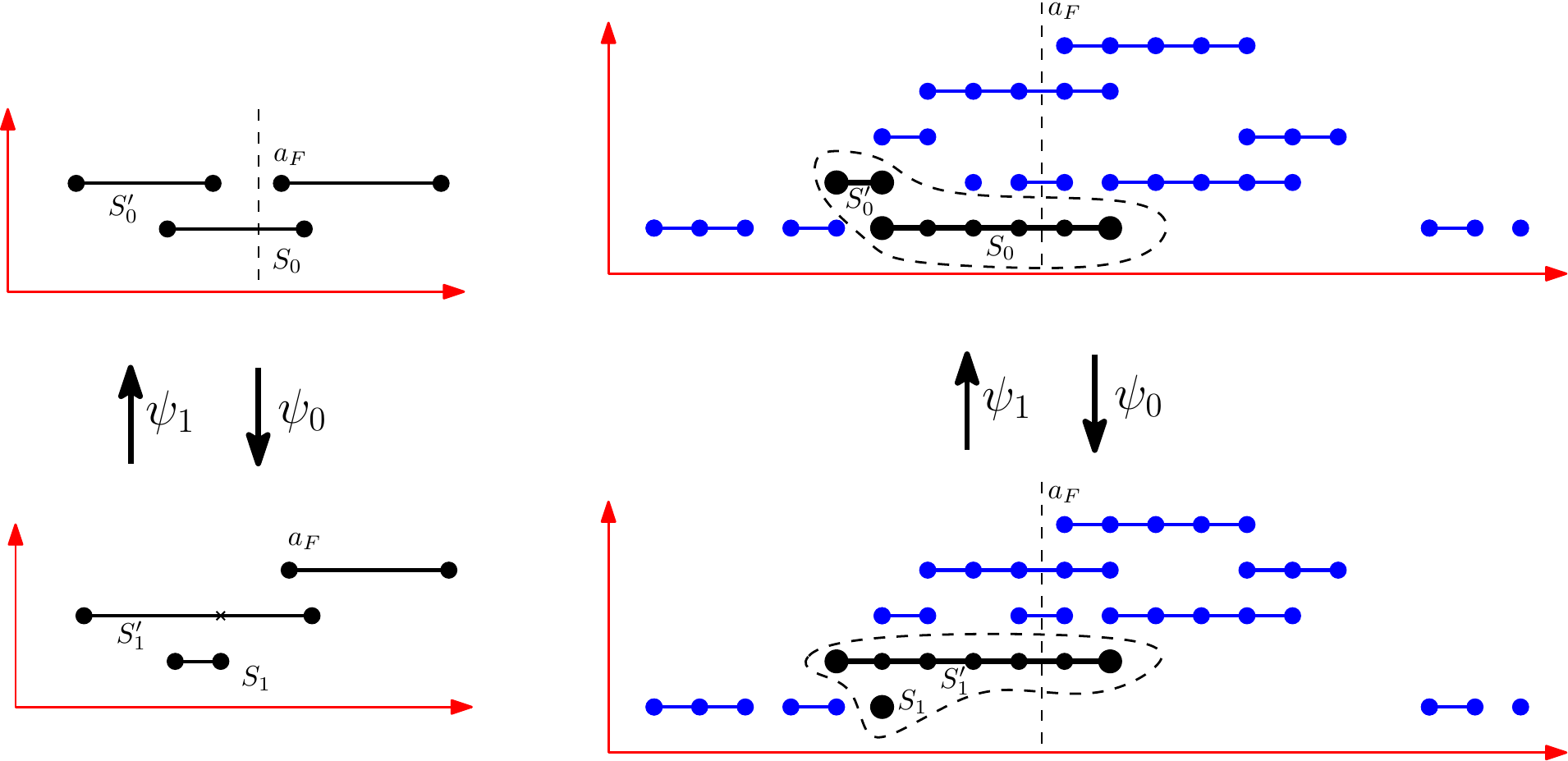}
\caption{Examples of the functions $\psi_i$: a simple case (left) and a more substantial one (right).  \label{fig:involution_psi}}
\end{figure}

 Suppose $F$ is in $\mathcal{W}_0$, and let $S_0(F)=[a_0,b_0]$ and $S'_0(F)=[a'_0,b'_0]\in \min(F')$ be as in Lemma~\ref{lem:CharacterizationW}. Loosely speaking, $\psi_0(F)$ is obtained by replacing these two pieces in $F$ by $[a_0,b'_0]$ and $[a'_0,b_0]$ respectively. Precisely, let $A=\Min(F)\setminus [a_0,b_0]$ and $B=F'\setminus [a'_0,b'_0]$. Then one has the decomposition $F=A*[a_0,b_0]*[a'_0,b'_0]*B$ and we set
\[\psi_0(F):=A*[a_0,b'_0]*[a'_0,b_0]*B.\]

Suppose $F$ is in $\mathcal{W}_1$, and let $S_1(F)=[a_1,b_1]$ and  $S'_1(F)=[a'_1,b'_1]\in \min(F')$ as in Lemma~\ref{lem:CharacterizationW}. Loosely speaking, $\psi_1(F)$ is obtained by replacing these two pieces in $F$ by $[a_1,b'_1]$ and $[a'_1,b_1]$ respectively. Precisely, let $A=\Min(F)\setminus [a_1,b_1]$ and $B=F'\setminus [a'_1,b'_1]$. Then one has the decomposition $F=A*[a_1,b_1]*[a'_1,b'_1]*B$ and we set 
\[\psi_1(F):=A*[a_1,b'_1]*[a'_1,b_1]*B.\]

\begin{Lemma}
\label{lem:psi_bijective}
The function $\psi_0$ is a bijection from $\mathcal{W}_0$ to $\mathcal{W}_1$ whose inverse is $\psi_1$.
\end{Lemma}

\begin{proof}
Let $F$ be in $\mathcal{W}_0$, and write $S_0(F)=[a_0,b_0]$ and $S'_0(F)=[a'_0,b'_0]$.  Then the minima of $\psi_0(F)$ are the same as those of $F$ except that $[a_0,b'_0]$ replaces $[a_0,b_0]$. From this it follows that $\psi_0(F)$ belongs to $\mathcal{W}$, and that it is of type $1$ with $S_1(\psi_0(F))=[a_0,b'_0]$ and $S'_1(\psi_0(F))=[a'_0,b_0]$. 

Now given $F$ in $\mathcal{W}_1$, write $S_1(F)=[a_1,b_1]$ and $S'_1(F)=[a'_1,b'_1]$. Then a similar reasoning shows that $\psi_0(F)$ is in $\mathcal{W}_0$ and satisfies $S_0(\psi_0(F))=[a_1,b'_1]$ and $S'_0(\psi_0(F))=[a'_1,b_1]$. 

It is then immediate that $\psi_1$ and $\psi_0$ are inverse to one another, given their definitions.
\end{proof}

The bijections $\psi_i$  satisfy $v(F)=v(\psi_i(F))$ for  $F\in\mathcal{W}_i$: indeed they preserve the number of pieces, total length of the segments and sum of the right endpoints. Moreover, one has $|\Min(F)|=|\Min(\psi_i(F))|$ and so $(-1)^{|U_1(F)|}=-(-1)^{|U_1(\psi_i(F))|}$. Therefore
\[
\sum_{F\in\mathcal{W}_0}v(F)(-1)^{|U_1(F)|}=-\sum_{F\in\mathcal{W}_1}v(F)(-1)^{|U_1(F)|},
\]
  and this proves~\eqref{eq:technical_vanishing}.\qee

\subsection{PPPs and  $321$-avoiding affine permutations}
\label{subsec:BackTo321}
%%%%%%%%%%%%%%%%%%%%
It is now time to link PPPs and $321$-avoiding affine permutations,
and see how one can essentially use Theorem~\ref{thm:enumPPP} --or
more precisely a variant given in Lemma~\ref{lem:calculFG}-- to prove
Theorem~\ref{thm:A-tA}. For completeness, we will also give the proof of the known finite case at the end.\\

Let $(P,c)$ be a PPP associated with the sequence $(a_i,b_i)_{1\leq
i\leq m}$ in $\tilde{\mathbf S}$ so that $m$ is the width of $P$.

\begin{Definition}
A \emph{marked PPP} is a triple $(P,c,j)$, where $(P,c)$ is a PPP and
$j\in \mathbb{N}$ satisfies $a_1\leq j\leq b_1$.
\end{Definition}

In this definition the mark $j$ must be considered graphically as a position on the
first column of $P$. In Figure~\ref{fig:PPPtoAAD}, the chosen position $j$ is denoted by an arrow, and the other possible choices by black squares, naturally indexed by $a_1,a_1+1,\ldots,b_1$ from bottom to top.

We define a {\em weak PPP}, as a PPP where we allow columns of
height zero, and adjacent columns may be incident in a corner, see Figure~\ref{fig:PPPtoAAD}, center. These are bijectively encoded by sequences $(a_i,b_i)_{1\leq i\leq m}$ in $\tilde{\mathbf S}$ where we allow the parts $a_i$ and $b_i$ to be zero.
\smallskip
 
\begin{figure}[!ht]
\includegraphics[width=\textwidth]{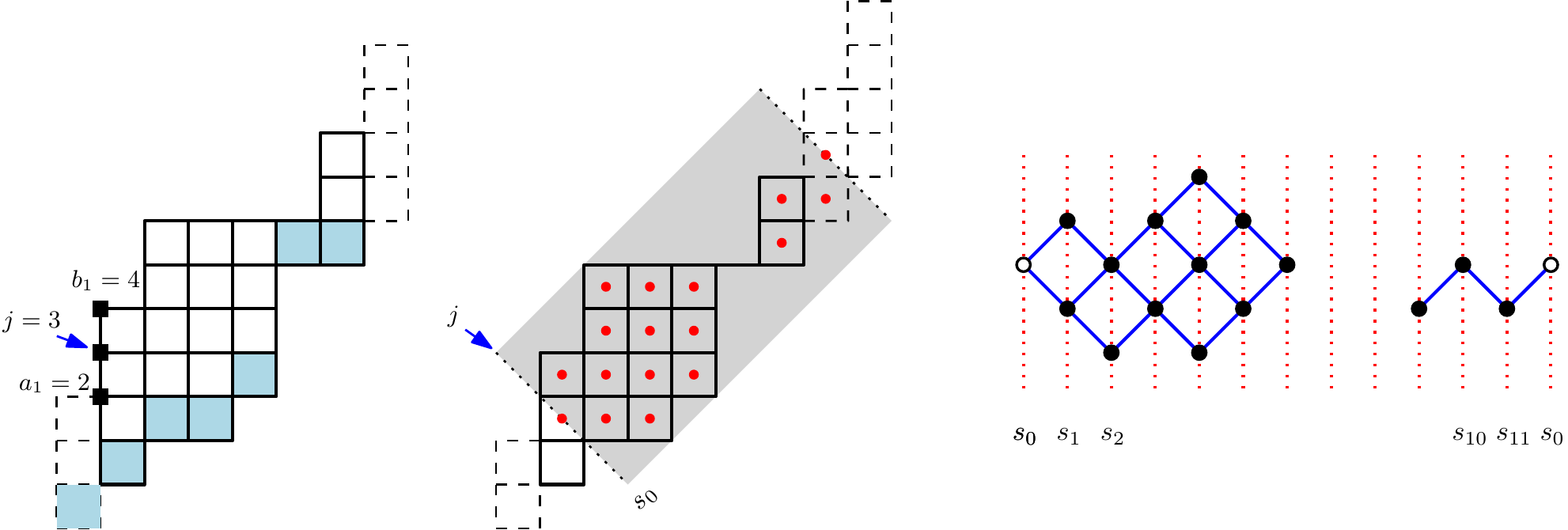}
\caption{From a marked PPP to an affine alternating
diagram.\label{fig:PPPtoAAD}}
\end{figure}

In what follows,  we associate to any marked PPP  an affine alternating diagram. The construction we give is an extension of Viennot's original insight for PPs, see Remark~\ref{rem:PPto321}. 

\smallskip

Given a marked PPP $(P,c,j)$, remove the bottom cell of each column of $P$ to obtain a weak PPP. Rotate it by 45 degrees clockwise. Replacing the cells by points, we get a set of points occurring on $n:=m+\sum_i(b_i-a_i)$ vertical lines, repeated periodically. This parameter $n$ is usually called the \emph{half-perimeter} of $P$. Now, we label by $s_0$ the points occurring on the vertical line indicated by $j$, by $s_1$ the ones on its right, and so on. Since points on adjacent vertical lines alternate from bottom to top, the resulting picture resembles an element of $\widetilde{\mathcal{D}}(n)$.

To check that we indeed have an element in $\widetilde{\mathcal{D}}(n)$ we must ensure that the resulting object actually represents a poset. For this, say that a PPP $(P,c)$ is \emph{rectangular} if $a_i=b_i=M$ for all $i$ where $M$ is a constant (recall that $c=a_1$); equivalently, $P$ is a rectangular polyomino and $c=b_1$. By the construction above, one can check that these rectangular PPPs become the excluded alternating diagrams of Proposition~\ref{prop:caracterisation_diagrams}(3) showed in Figure~\ref{fig:rectangular_shape}, right. It is easy to see that in all other cases this construction produces a genuine affine alternating diagram.

\begin{Proposition}
\label{prop:mPPPtoAAD}
Let $n\geq 1$. The preceding construction is a bijection between
non-rectangular, marked PPPs of half-perimeter $n$ and affine
alternating diagrams of size $n$.
\end{Proposition}

\begin{proof}
We describe graphically the inverse construction, which we illustrate in Figure~\ref{fig:AADtoPPP}.
Let $D$ be an affine alternating diagram in $\widetilde{\mathcal{D}}(n)$. Replace all its points by cells, see Figure~\ref{fig:AADtoPPP}, left, and rotate this picture 45 degrees counterclockwise. If all labels occur in $D$, this gives a PPP\footnote{Notice that because $D$ is not of the form described in Proposition~\ref{prop:caracterisation_diagrams}(2) (and showed in Figure~\ref{fig:rectangular_shape}, left), the rotation will not result in infinite columns.} up to the choice of the first column (see below). Otherwise, we get a collection of PPs. In this case, if there are $k$ missing labels between two points in $D$, join the corresponding PPs by $k-1$ empty columns, see Figure~\ref{fig:AADtoPPP}, center, for an illustration.

We get a weak PPP up to the fact that we must choose its first column. For this, note that a marking $j$ on the north-west boundary  is determined by recording the diagonal corresponding to $s_0$. The first column is then chosen to be the one to the right of the marking. To get a marked PPP, we simply add a cell at the bottom of each column, see Figure~\ref{fig:AADtoPPP}, right, and we record the marking $j$.

We skip the verification that we indeed described the desired inverse, since it is essentially similar to Viennot's construction, see Remark~\ref{rem:PPto321}.

\end{proof}

\begin{figure}[!ht]
\includegraphics[width=\textwidth]{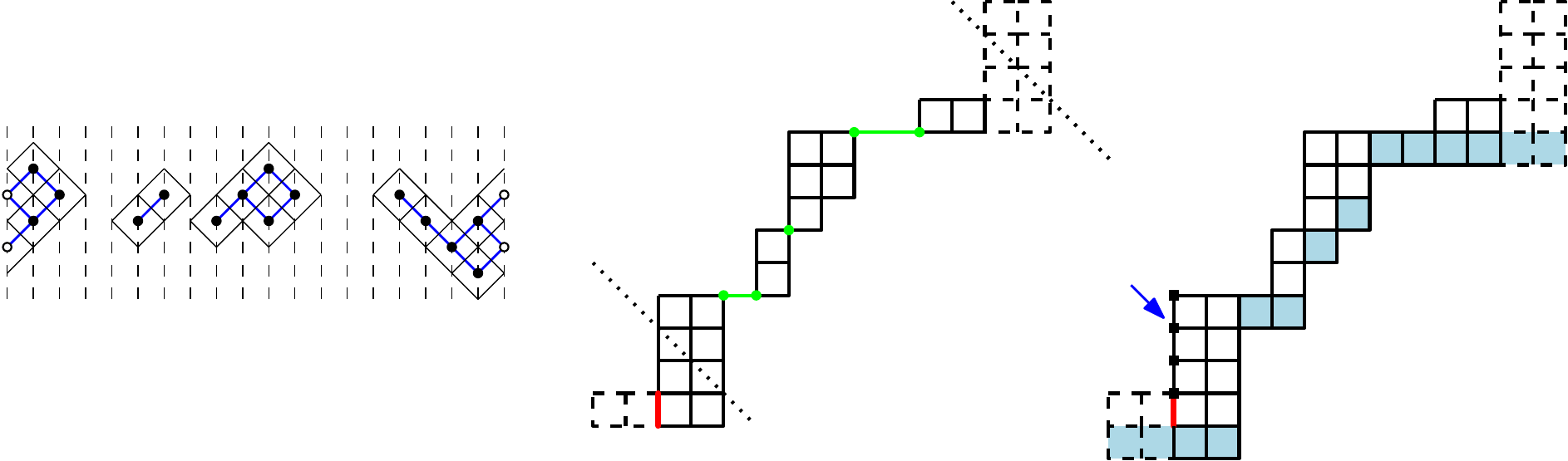}
\caption{From an affine alternating diagram to a marked
PPP.\label{fig:AADtoPPP}}
\end{figure}

We know that PPPs are in bijection with heaps in $\Ht$, see Proposition \ref{prop:bijcyclicheap}. This induces a bijection between marked PPPs and $\Ht^*$, the set of heaps in $\Ht$ with a distinguished point on their rightmost
maximal segment. Taking weights into account, Theorem~\ref{thm:bijectionSD},  \eqref{eq:polyominoes_as_heaps} left, and  Proposition \ref{prop:mPPPtoAAD} imply

\begin{equation}\label{atildewithH}
\tS(x,q)=\sum_{H\in
\Ht^*}x^{\ell(H)+|H|}q^{e(H)-|H|}-\sum_{n\geq1}\frac{x^nq^n}{1-q^n}.
\end{equation}

The subtracted term enumerates rectangular PPPs. To handle the first sum in this expression, one needs to consider marked versions of Lemmas~\ref{lem:imagephi} and
\ref{lem:psi_bijective}, as will be done in the next result.
\begin{Lemma}\label{lem:calculFG}
We have:
\[
\sum_{H\in \Ht^*}x^{\ell(H)}y^{|H|}q^{e(H)}=-x\frac{\partial_x N(x,y,q)}{N(x,y,q)}.
\]
\end{Lemma}
\begin{proof}

We examine the image $\phi(H,T)$ where $H\in\Ht^*$ and
$T\in\mathcal{T}$, as in Lemma~\ref{lem:imagephi}, where the marked
segment in $H$ naturally becomes a marked segment in $F$. Each step in
the proof of the latter is still valid by replacing $\Ht$ by $\Ht^*$, 
noting that when $F$ is trivial, any of its segments can be marked,
and when it is not trivial the marked segment is $[a_F,b_F]$.
Therefore we get
\begin{multline*}
\left( \sum_{T\in \mathcal{T}} (-1)^{|T|}v(T)\right)\left(\sum_{H\in
\Ht^*} v(H)\right)
= \sum_{F\in\mathcal{T}}\ell(F)(-1)^{|F|-1}v(F)\\
+\sum_{F, F\setminus U_1(F)\in\Ht^*}v(F)\left(\sum_{U_1(F)\subseteq
U\subseteq U_2(F)}(-1)^{|U|}\right).
\end{multline*}
The first sum on the right-hand side is equal to $-x\partial_x N(x,y,q)$. It remains to see that again the second sum vanishes. This is still immediate when $U_1(F)\neq U_2(F)$. Otherwise, applying the bijections $\psi_0$ and $\psi_1$ of
Lemma~\ref{lem:psi_bijective} and noting that the marked segment
$[a_F,b_F]$ is never involved in their constructions, we derive the
result.
\end{proof}

Summarizing, Lemma~\ref{lem:calculFG} and~\eqref{atildewithH} together yield
\begin{eqnarray*}
\tS(x,q)&=&-x\frac{\partial_xN(x,x/q,q)}{N(x,x/q,q)}-\sum_{n\geq1}\frac{x^nq^n}{1-q^n}\\
&=&-x\frac{J'(x)}{J(x)}-\sum_{n\geq1}\frac{x^nq^n}{1-q^n},
\end{eqnarray*}
which is the desired expression from Theorem~\ref{thm:A-tA}.

\begin{Remark}\label{rem:PPto321}
The case of PPs, due to Viennot (see~\cite{ViennotWeb}), corresponds to marked PPPs of the special form $(P,1,1)$. These are in bijection with diagrams having no point labeled $s_0$, \emph{i.e.} finite diagrams, by restriction of the correspondence above. This gives the identity:
\begin{equation}
\label{eq:SPP}
S(x,q)=PP(x,x/q,q).
\end{equation}
To prove the first claim of Theorem~\ref{thm:A-tA}, namely $S(x,q)=\frac 1
{1-xq}\frac{\JJ(xq)}{\JJ(x)}$, where $J$ is defined in~\eqref{J-def}, it suffices to combine ~\eqref{eq:enumPP} and~\eqref{eq:SPP}, since we have that $N(x,x/q,q)=J(x)$ and $\hat{N}(x,x/q,q)=-xJ(xq)/(1-xq)$ by comparing their expressions in
\eqref{eq:trivialheaps} and~\eqref{J-def}.
\end{Remark}

%%%%%%%%%%%%%%%%%%%%%%%%%%%%%%%%%%%%%%%%%%%%%%%%%%%%%%%%%%%%%%%%%%%%%%%%%%%
\section{Further questions}\label{sec:furtherquestions}
%%%%%%%%%%%%%%%%%%%%%%%%%%%%%%%%%%%%%%%%%%%%%%%%%%%%%%%%%%%%%%%%%%%%%%%%%%%

\noindent{\bf Other types.} A natural question arises in view of Section~\ref{sec:heaps-monomeres-dimeres}, regarding FC elements in other types. Indeed, in~\cite{BBJN},  the bivariate generating functions \eqref{eq:bivariate} were explicitly computed for FC elements in all classical finite and affine types, while a complete description in terms of (alternating) diagrams was given in~\cite{BJN-long}. Although the formulas are not as nice as in Theorem~\ref{thm:A-tA}, there are two interesting cases which can presumably be treated through the approach by heaps of monomers and dimers: the alternating FC elements of finite type $B$ and affine type $\widetilde{C}$. For instance, in type $B$, these alternating FC elements are called FC top elements of $B_n$ in~\cite{St3}. They  are a subfamily of FC elements (the remaining type $B$ FC elements are called left-peaks in~\cite{BJN-long}), and their generating function is given in~\cite{BBJN} by:
\begin{equation}\label{typeBalt}
\frac {1}{\JJ(x)}\,\sum_{n\ge 0} \frac{x^n q^{n+1\choose 2} }{(xq;q)_{n}}\sum_{k=0}^{n}\frac{(-1)^k}{(q;q)_k},
\end{equation}
where $\JJ(x)$ is defined in~\eqref{J-def}. These FC elements are in correspondence with walks on the graph of Figure~\ref{fig:Graphe_Motzbic} starting at any vertex and ending at vertex $0$.  In terms of heaps, Theorem~\ref{thm:bijHeapsCycles} yields a bijection with pyramids of monomers and dimers in the set $\cH(\cP_{md}^*,\cC)$ (see Section~\ref{sec:heaps-monomeres-dimeres}), except that their  unique  maximal piece is an additional piece of the form $[0;i]$, with weight $x^iq^{\bi{i+1}{2}}$, for a nonnegative integer $i$. Thanks to the Inversion Lemma, the generating function of these objects is given by
$$\frac{1}{j(x)}\sum_{i\geq0}x^iq^{\bi{i+1}{2}}h(xq^{i+1})=\frac{1}{(xq;q)_\infty\JJ(x)}\sum_{i,j\geq0}x^iq^{\bi{i+1}{2}}(-x)^jq^{\bi{j}{2}}\frac{(xq^{i+j+1};q)_\infty}{(q;q)_j},$$
where $h(x)$ and $j(x)$ are defined in~\eqref{h-def} and~\eqref{j-def}, respectively. By setting $n=i+j$ and $k=j$ in the above double summation, we derive~\eqref{typeBalt}.

For affine type $\widetilde{C}$, the walks we have to count are the ones starting and ending  at any vertex on the graph of Figure~\ref{fig:Graphe_Motzbic}. Therefore it would be interesting to find the corresponding formula by using heaps of monomers and dimers. Of course this approach should also give the generating function of FC involutions in types $B$ and $\widetilde{C}$.
\\

\noindent{\bf Pyramids and PPPs.} 
Recall from Section~\ref{sec:heaps} the set $\Pi\subset \H$ of pyramids. An application of Corollary~\ref{Cor:enumpyramids} in the context of Section~\ref{sec:parallelogram} gives immediately 
\[\sum_{H\in \Pi}v(H)=-y\frac{\partial_y N(x,y,q)}{N(x,y,q)}\]
which is also equal to $\sum_{H\in \Ht}v(H)=PPP(x,y,q)$ by \eqref{eq:polyominoes_as_heaps} left, and Theorem~\ref{thm:enumPPP}.
 
  A weight-preserving bijection between the sets $\Ht$ and $\Pi$ still eludes us though; equivalently, one would like a direct way of encoding PPPs as pyramids. This would simplify the proof of their enumeration.
  \\

\noindent{\bf Involutions and PPPs.} In Section~\ref{sec:heaps-monomeres-dimeres}, we gave a bijective proof of Theorem~\ref{thm:A-tA} which also yielded Theorem~\ref{thm:A-tA-inv}, \emph{i.e.}, the case of involutions. The approach by PPPs in Section~\ref{sec:parallelogram} does specialize nicely in the same way. It is interesting to look for an encoding of $321$-avoiding affine involutions (or, equivalently, of self-dual diagrams) in the same spirit as PPPs, which would give an alternative proof of Theorem~\ref{thm:A-tA-inv}.
\\

\noindent{\bf Acknowledgements.} The authors thank Mireille Bousquet-M\'elou for helpful and inspiring discussions at the start of this project. 

%%%%%%%%%%%%%%%%%%%%%%%%%%%%%%%%%%%%%%%%%%%%%%%%%%%%%%%%%%%%%%%%%%%%%%%%%%%
%%%%%%%%%%%%%%%%%%%%%%%%%%%%%%%%%%%%%%%%%%%%%%%%%%%%%%%%%%%%%%%%%%%%%%%%%%%

%\bibliographystyle{plain}
%\bibliography{fc}

\begin{thebibliography}{10}

\bibitem{aval-ppp}
J.-C. Aval, A.~Boussicault, P.~Laborde-Zubieta, and M.~P\'etr\'eolle.
\newblock Generating series of periodic parallelogram polyominoes.
\newblock \href{https://arxiv.org/abs/1612.03759}{arXiv:1612.03759}, 2016.

\bibitem{barcucci}
E.~Barcucci, A.~Del~Lungo, E.~Pergola, and R.~Pinzani.
\newblock Some permutations with forbidden subsequences and their inversion
  number.
\newblock {\em Discrete Math.}, 234(1-3):1--15, 2001.

\bibitem{BBJN}
R.~Biagioli, M.~Bousquet-M\'elou, F.~Jouhet, and P.~Nadeau.
\newblock Length enumeration of fully commutative elements in finite and affine
  {C}oxeter groups.
\newblock \href{https://arxiv.org/abs/1612.07591}{arXiv:1612.07591}, 2016.

\bibitem{BJN-inv}
R.~Biagioli, F.~Jouhet, and P.~Nadeau.
\newblock Combinatorics of fully commutative involutions in classical {C}oxeter
  groups.
\newblock {\em Discrete Math.}, 338(12):2242--2259, 2015.

\bibitem{BJN-long}
R.~Biagioli, F.~Jouhet, and P.~Nadeau.
\newblock Fully commutative elements in affine and finite {C}oxeter groups.
\newblock {\em Monatsh. Math.}, 178(1):1--37, 2015.

\bibitem{BBJN-bij}
R.~Biagioli, F.~Jouhet, and P.~Nadeau.
\newblock 321-{A}voiding affine permutations, heaps, and periodic parallelogram
  polyominos.
\newblock In {\em Proceedings of GASCom 2016}, volume~59 of {\em Elect. Notes
  in Discrete Math.}, pages 115--130, 2017.

\bibitem{BJS}
S.~C. Billey, W.~Jockusch, and R.~P. Stanley.
\newblock Some combinatorial properties of {S}chubert polynomials.
\newblock {\em J. Algebraic Combin.}, 2(4):345--374, 1993.

\bibitem{bjorner-brenti-book}
A.~Bj{\"o}rner and F.~Brenti.
\newblock {\em Combinatorics of {C}oxeter groups}, volume 231 of {\em Graduate
  Texts in Mathematics}.
\newblock Springer, New York, 2005.

\bibitem{mbm-viennot}
M.~Bousquet-M{\'e}lou and X.~G. Viennot.
\newblock Empilements de segments et {$q$}-\'enum\'eration de polyominos
  convexes dirig\'es.
\newblock {\em J. Combin. Theory Ser. A}, 60(2):196--224, 1992.

\bibitem{boussicault-laborde}
A.~Boussicault and P.~Laborde-Zubieta.
\newblock Periodic parallelogram polyominoes.
\newblock In {\em Proceedings of GASCom 2016}, volume~59 of {\em Elect. Notes
  in Discrete Math.}, pages 177--188, 2017.

\bibitem{cartierfoata}
P.~Cartier and D.~Foata.
\newblock {\em Probl\`emes combinatoires de commutation et r\'earrangements}.
\newblock Lecture Notes in Mathematics, No. 85. Springer-Verlag, Berlin-New
  York, 1969.

\bibitem{Graham}
J.~Graham.
\newblock {\em Modular Representations of Hecke Algebras and Related Algebras}.
\newblock PhD thesis, University of Sydney, 1995.

\bibitem{Gre321}
R.~M. Green.
\newblock On 321-avoiding permutations in affine {W}eyl groups.
\newblock {\em J. Algebraic Combin.}, 15(3):241--252, 2002.

\bibitem{HagiwaraAtilde}
M.~Hagiwara.
\newblock Minuscule heaps over {D}ynkin diagrams of type {$\widetilde{A}$}.
\newblock {\em Electron. J. Combin.}, 11(1):Research Paper 3, 20, 2004.

\bibitem{Lam}
T.~Lam.
\newblock Affine stanley symmetric functions.
\newblock {\em Amer. J. Math.}, 128(6):1553--1586, 2006.

\bibitem{LusztigTransactions}
G.~Lusztig.
\newblock Some examples of square integrable representations of semisimple
  {$p$}-adic groups.
\newblock {\em Trans. Amer. Math. Soc.}, 277(2):623--653, 1983.

\bibitem{Postnikov}
A.~Postnikov.
\newblock Affine approach to quantum schubert calculus.
\newblock {\em Duke Math. J.}, 128(3):473--509, 2005.

\bibitem{Shi}
J.~Y. Shi.
\newblock {\em The {K}azhdan-{L}usztig cells in certain affine {W}eyl groups},
  volume 1179 of {\em Lecture Notes in Mathematics}.
\newblock Springer-Verlag, Berlin, 1986.

\bibitem{Stanley}
R.~P. Stanley.
\newblock On the number of reduced decompositions of elements of {C}oxeter
  groups.
\newblock {\em European J. Combin.}, 5:359--372, 1984.

\bibitem{St3}
J.~R. Stembridge.
\newblock The enumeration of fully commutative elements of {C}oxeter groups.
\newblock {\em J. Algebraic Combin.}, 7(3):291--320, 1998.

\bibitem{Viennot1}
X.~G. Viennot.
\newblock Heaps of pieces. {I}. {B}asic definitions and combinatorial lemmas.
\newblock In {\em Graph theory and its applications: {E}ast and {W}est
  ({J}inan, 1986)}, volume 576 of {\em Ann. New York Acad. Sci.}, pages
  542--570. New York Acad. Sci., New York, 1989.

\bibitem{ViennotWeb}
X.~G. Viennot.
\newblock Course {IMS}c {C}hennai, {I}ndia.
\newblock \href{http://coursimsc2017.xavierviennot.org/}{Chapter 6a}, 2017.

\end{thebibliography}

\end{document}